\documentclass[12pt]{amsart}
\usepackage{graphicx,subfigure,epsfig}
\usepackage{subfigure}
\usepackage{epstopdf}
\usepackage{mathtools}

\usepackage{amssymb,amsmath,latexsym}
\usepackage{caption}
\usepackage{float}
\usepackage[usenames, dvipsnames]{color}
\usepackage{bm}
\usepackage{enumitem}
\newtheorem{theorem}{Theorem}[section]
\newtheorem{lemma}[theorem]{Lemma}

\newtheorem{example}[theorem]{Example}

\newtheorem{prop}[theorem]{Proposition}
\newtheorem{remark}[theorem]{Remark}

\newtheorem{cor}{Corollary}

\newtheorem{problem}{Problem}
\usepackage{placeins}
\usepackage{graphicx}

\begin{document}
\thispagestyle{empty}

\title{On residually nilpotence of  group extensions}

\author{{V. G. Bardakov, M. V. Neshchadim, and O. V. Bryukhanov}}%

\maketitle {\small}

\begin{abstract}
We study the following question:
under what conditions extension of one residually nilpotent group by another
residually nilpotent group is residually nilpotent?
We prove  some sufficient conditions under which this extension  is residually nilpotent.
Also, we  study this question for semi-direct
products and,
in particular, for extensions of free group by infinite cyclic group: $F_n \rtimes_{\varphi} \mathbb{Z}$. We find conditions under which this group is residually nilpotent,
find conditions under which this group has long lower central series. In particular, we prove that for $n=2$
the length of the lower central series of $F_n \rtimes_{\varphi} \mathbb{Z}$ is equal to 2, $\omega$ or $\omega^2$.

\end{abstract}


\tableofcontents

\section{Introduction}

A group $G$ is said to be a {\it residually nilpotent} if, for every non-identity element $x \in G$, there exists a normal subgroup $N$, depending on $x$, such that $x \not\in N$ and the quotient group $G / N$ is nilpotent (see, for example \cite{Mal-1}).
One of the first results on residually nilpotent groups belongs to Magnus who proved that  a free group
 is residually nilpotent (a simple proof of this fact can be found
in \cite[Theorem 14.2]{KM}). Malcev \cite{Mal-1} studied the conditions under which
 the free product of residually nilpotent groups to be residually nilpotent.
In particular, he established that the  free product of nilpotent groups is residually nilpotent if and only if all these groups do not contain elements of infinite $p$-height for some prime $p$, or all these groups are groups without torsion. Recall that an element $g \in G$ is called an {\it element of infinite $p$-height},  if, given any two natural number $i$ and $j$, there exist elements $x\in G$ , $y \in \gamma_i(G)$ such that $x^{p^j} = g y$. From this result follows that the groups $\mathbb{Z} * \mathbb{Z}_p$, $\mathbb{Z}_{p^n} * \mathbb{Z}_{p^m}$ are residually nilpotent. On the other side,  $\mathbb{Z}_2 * \mathbb{Z}_3$ is not residually nilpotent.

No general approaches are available by now even to solving the
problem of residual nilpotency for one-relator groups: only partial
results are available.  Baumslag \cite{Bau} studied
in  one special class of one-relator groups.
Loginova obtained in \cite{L} a similar result for a larger class of
groups. The paper \cite{McC} gives a characterization of residually
nilpotent one-relator groups with non-trivial
center. Mikhailov \cite{M} constructed one-relator group with the  lower
central series of length $\omega^2$. This answers a problem of  Baumslag.
It is established in \cite{M-1} that a central extension of one-relator
groups is residually nilpotent if and
only if so is the original group.

Malcev proved \cite{Mal-2} (also see \cite[Theorem 51.2.1]{Mer})
that if $G$ is a finitely generated subgroup of
$\mathrm{GL}_n(F)$ where $F$ is some field of characteristic 0,
then $G$ contains some subgroup of finite index which is residually
$p$-finite for almost all prime $p$.
In \cite{BM} was  established a sufficient condition for
the residual $p$-finiteness of finitely generated
subgroups of $\mathrm{GL}_n(F)$.
As a corollary to this result was established
 the residual 2-finiteness of some link groups, including those of the Whitehead link and
Borromean links. On the other side, there exists a 2-component 2-bridge link whose group is not residually nilpotent.
Also in \cite{BM} it was proved that each
link is a sublink of some link whose group is residually 2-finite. Residually nilpotence of groups of virtual knots is studied in \cite{BMN}--\cite{BNN}.
Residually nilpotence of the fundamental groups of some 3-manifolds is studied in \cite{Br}.

Linear properties of group extensions are studied in \cite{BBr, Br-1}
Semi-direct product is a particular case of HNN-extension. Residually nilpotence of HNN-extensions is studied in \cite{Mol1} and in particular, for Baumslag-Soliter groups and in generalized Baumslag-Solitar groups in \cite{Mol2, Mol3, BN-4, S}.

Survey on residually nilpotent groups can be found in the book of Mikhailov-Passi \cite{MP}.

Gruenberg \cite{G} introduced root classes of groups and
proved:
If $\mathcal{P}$ is a root class and $H$ is a normal subgroup
of $G$ such that $G / H$ lies in $\mathcal{P}$ and $H$ is residually $\mathcal{P}$,
then $G$ is residually $\mathcal{P}$.
The following classes of groups are root classes:
solvabel groups, finite groups, class of finite $p$-groups for some  prime $p$.
Class of nilpotent groups  is not a root class.

In the present paper we
 study the following question:
under what conditions extension of one residually nilpotent group by another
residually nilpotent group is residually nilpotent? In other words,  if
\begin{equation} \label{ses}
1 \to A \to G \to B \to 1
\end{equation}
is a short exact sequence in which $A$ and $B$ are residually nilpotent, under what conditions $G$ is residually nilpotent?

The residually nilpotence  groups of the form $\mathbb{Z}^n \rtimes_{\varphi} \mathbb{Z}$, $ \varphi\in\mathrm{Aut}(\mathbb{Z}^{n})$,
are studied by  Aschen\-bren\-ner and   Friedl \cite{AF}.
They have found a criteria of residually nilpotence and residually $p$-nilpotence for groups of this  type.

The paper is organized as follows.

Section \ref{s2} collects some basic definitions: the  lower central series, the lower $p$-series, root class. Also,  we formulate some well known facts about commutator subgroups and their products, result of Gruenberg.
We recall the result  of Aschen\-bren\-ner and   Friedl \cite{AF}
on the  residually nilpotence and residually $p$-nilpotence for groups of the form
$\mathbb{Z}^{n}\rtimes_{\varphi}\mathbb{Z}, \varphi\in\mathrm{Aut}(\mathbb{Z}^{n})$.

In Section \ref{s3} we study general extension of one group by another one. We prove  some sufficient conditions under which $G$ in (\ref{ses}) is residually nilpotent.
In $G$ the subgroup $B$ acts on $A$. Then we can consider the abelianization $\overline{A} = A / A'$ as left $\mathbb{Z}[B]$-module.
We prove that if $A$ is residually nilpotent, $B$ is nilpotent and $\overline{A}$ is a nilpotent $\mathbb{Z}[B]$-module, i.e.
 $\Delta_B^N \overline{A} = 0$ for some natural $N$, then $G$ is residually nilpotent (see Theorem \ref{t3.1}).

Similar theorem can be proved if we take a residually $p$-nilpotent group $A$ instead residually nilpotent, where $p$ is a prime number. If $A$ is residually $p$-nilpotent group,
$B$ is nilpotent and for some $N \geq 1$ the inclusion $\Delta_B^N \overline{A} \subseteq p \overline{A}$ holds, then $G$ is residually nilpotent (see Theorem \ref{t3.3}).

As was proved by Falk and Randel  \cite{FR1}, if $A$ and $B$ are residually nilpotent groups,
$G = A \rtimes_{\varphi} B$ is a semi-direct product, where the action of $B$ on $A$ is trivial
by modulo of the commutator subgroup $A' = \gamma_2 (A)$, then
 $G$ is residually nilpotent.

In Section \ref{s4} we  prove the following  generalization of this result.
If $A$ is residually $p$-nilpotent for some prime number $p$, $B$
is residually nilpotent, and $G = A \rtimes_{\varphi} B$,
where the action of $B$ on $A$ is trivial by modulo  $\gamma_2^{(p)} (A)$, then
$$
\gamma_k(G) = \gamma^{(p)}_k(A) \gamma_k(B),~~k \geq 1,
$$
and $G$ is residually nilpotent
(see Theorem \ref{t5.1}).

In \cite{GM} for a semi-direct product $G = A \rtimes_{\varphi} B$ was given some description of the terms of the lower central series.
In Section \ref{s4} we give simpler description of $\gamma_2(G)$  and $\gamma_2(G)$.
Also, in \cite{GM} was found  the lower central series of the pure braid group $P_2(\mathbb{K})$ of the  Klein bottle
 $\mathbb{K}$. Using our approach we prove
that $P_2 (\mathbb{K})$ is residually nilpotent.

In Section \ref{FnZ} we consider  free-by-infinite cyclic group $G=F_n \rtimes_{\varphi} \mathbb{Z}$.
 Baumslag \cite{Bau-1} proved  that groups of this type are residually finite.
The situation with residually nilpotence is more complicated. There exist finitely generated cyclic extensions of a free group which is not residually nilpotent. Moreover, Mikhailov \cite{M}, answering  on a question of Baumslag,  constructed a cyclic extension of the free 2-generated group which has the length of the lower central series equal to $\omega^2$.

We define two families of groups:
$\overline{G}_k= (F^{ab}_{n})^{\otimes k}\rtimes_{\overline{\varphi}_{k}}\mathbb{Z}$ and $\widehat{G}_k= \gamma_{i}(F_{n})/\gamma_{i+1}(F_{n})\rtimes_{\widehat{\varphi}_{i}}\mathbb{Z}$, $k = 1, 2, \ldots$, and prove that if
 $\overline{G}_k$ is residually nilpotent, then $\widehat{G}_k$ is residually nilpotent;
if $\overline{G}_k$ is residually $p$-finite, then $\widehat{G}_k$ is residually $p$-finite.

We find conditions under which $G$ has the short lower central series.
In Proposition \ref{FnZn} we prove that $\gamma_{2}(G) =\gamma_{\omega}(G) = F_n$ if and only if the matrix $[\overline{\varphi}] - E$ lies in $\mathrm{GL}_n(\mathbb{Z}).$ In particular, in this case, $G$ is not residually nilpotent.

Further, we describe conditions under which $G$ has the long lower central series and give a full prove of the following theorem (see Theorem \ref{Mres}):
If   $G$ is not residually nilpotent and  all groups  $\overline{G}_k$,   $k\geq1$,
are residually nilpotent,  then the length of the  lower central series of $G$ is equal to  $\omega^{2}$.

 This theorem was formulated in \cite{M} without proof.

In Section \ref{pFnZ} we study residually $p$-finiteness of group $G = F_{n}\rtimes_{\varphi}\mathbb{Z}$. The main result of this section is Theorem
\ref{Fpres}:
If all groups $\widehat{G}_i, i\geq1,$
are residually    $p$-finite, then  $G$ is residually    $p$-finite.

As corollary of this theorem we get Corollary \ref{Npres}:
Let $G = N_{n,c}\rtimes_{\varphi}\mathbb{Z}$, where $N_{n,c}$ be a free  nilpotent group of rank $n$ and class $c$.
If all groups
$$
\gamma_{i}(N_{n,c})/\gamma_{i+1}(N_{n,c})\rtimes_{\widehat{\varphi}_{i}}\mathbb{Z}, ~~~1\leq i\leq c,
$$
are residually  $p$-finite, then  $G$ is residually  $p$-finite. Here  $\widehat{\varphi}_{i}$ is the automorphism of  $\mathbb{Z}$-module
$\gamma_{i}(N_{n,c})/\gamma_{i+1}(N_{n,c})$ that is induced by  $\varphi$.

As was proved in \cite{M}, the residually nilpotence of $\overline{G}_k$, $k \geq 1$, does not imply the residually nilpotence  of $G = F_{n}\rtimes_{\varphi}\mathbb{Z}$.
Using  Theorem \ref{Fpres} it is possible to prove that the residually $p$-finiteness of $\overline{G}_k$, $k \geq 1$, implies  the residually $p$-finiteness  of
$G = F_{n}\rtimes_{\varphi}\mathbb{Z}$.

The following theorem (see Theorem \ref{Fdres1}) gives a simple way to prove the residually nilpotence of free-by-infinite cyclic groups.
Let $G = F_{n}\rtimes_{\varphi}\mathbb{Z}$, $\varphi\in\mathrm{Aut}(F_{n})$ and $[\overline{\varphi}]\in\mathrm{GL}_{n}(\mathbb{Z})$. If all eigenvalues of  $[\overline{\varphi}]$ are integers, then  $G$ is residually nilpotent.

Section \ref{s8} is dedicated to the groups of the form $G = F_2 {\rtimes}_\varphi \mathbb{Z}$. Using results of Sections  \ref{FnZ} and \ref{pFnZ} we prove a
criteria of the residually nilpotence of $G$ (see Theorem \ref{NAF2Z}).
Also, we prove in Theorem \ref{F2Zc} that
for the lower central series of $G = F_2 {\rtimes}_\varphi \mathbb{Z}$, $\varphi\in\mathrm{Aut}(F_{2})$, there exist only three possibilities:
\begin{itemize}
\item[a)]  $\gamma_{\omega}(G)=\gamma_{2}(G)$;
\item[b)]  $\gamma_{\omega}(G)=1$;
\item[c)]  $\gamma_{\omega^{2}}(G)=1$ and  the length of the lower central series is equal to $\omega^{2}$.
\end{itemize}

We prove  a criterion (see Theorem \ref{F2Zr1}) of residually nilpotence of  $G = F_2 {\rtimes}_\varphi \mathbb{Z}$ with $\det[\overline{\varphi}] = 1$.
In this case    $G$  is residually nilpotent if and only if  $\mathrm{tr}[\overline{\varphi}]\not\in \{1, 3 \}$.

As we seen, not any group of the form $F_{n} {\rtimes}_\varphi \mathbb{Z}$ is residually nilpotent. On the other side, Azarov \cite{Azar}  proved that any semi-direct products of a finitely generated  residually $p$-finite group by a residually $p$-finite group is virtually residually $p$-finite and hence, contains residually nilpotent subgroups of finite index. From this result follows that any group of the form  $F_{n} {\rtimes}_\varphi \mathbb{Z}$ is virtually residually nilpotent. Also, we prove that $F_2 {\rtimes}_\varphi \mathbb{Z}$
contains residually nilpotent subgroup of index 2 or 4.

\section{Preliminaries} \label{s2}

The following definitions  can be found in the paper of Malcev  \cite{Mal}.
Let $\mathfrak{C}$ be some class of groups. A group   $G$ is said
to be {\it  residually  $\mathfrak{C}$-group}
 or simply  $\mathfrak{C}$--{\it residual}, if for any non-identity  element
 $g \in G$ there exists a homomorphism  $\varphi$ of  $G$ to some group
 from  $\mathfrak{C}$ such that  $\varphi(g) \not= 1$. If
$\mathfrak{C}$ is the class of all finite groups, then  $G$ is called
{\it residually finite}.  If $\mathfrak{C}$ is the class of finite
$p$-groups, then $G$ is said to be {\it residually $p$-finite}.
If $\mathfrak{C}$ is the class of nilpotent groups, then   $G$ is  said to be
 {\it residually nilpotent}. If $\mathfrak{C}$ is the class of
 torsion-free nilpotent groups, then   $G$ is  said to be
 {\it torsion free residually nilpotent}. If $\mathfrak{C}$ is the class of
 nilpotent $p$--groups, then   $G$ is  said to be
 {\it  residually $p$-nilpotent}.

Let $G$ be a group  and let $x_1, x_2, \ldots$ be elements of $G$.
A {\it simple commutator} $[x_1, x_2, \ldots, x_n]$ of weight $n \geq 1$ is defined
inductively by setting
$$
[x_1] = x_1,~~[x_1, x_2] =x_1^{-1} x_2^{-1} x_1 x_2~\mbox{and}~[x_1, x_2, \ldots, x_n] = [[x_1, x_2, \ldots, x_{n-1}], x_n]~\mbox{for}~n \geq 3.
$$
Let  $A, B \subseteq G$ be two subsets of a group $G$,
then by  $[A, B]$ we shall denote the subgroup which is generated by all commutators
 $[a, b]$, $a \in A$, $b \in B$. Suppose that $P$ and $Q$ are subgroups of $G$.
 The subgroup of $G$ that is generated by left-normalized commutators $[p,q_1,\ldots,q_k]$,
 where $p\in P$, $q_1,\ldots,q_k \in Q$ we shall denote by $[P,{}_{k}Q]$.

 We will use the following commutator identities
 \begin{equation}
 [x, y]^{-1} = [y, x];
  \end{equation}
  \begin{equation}
 [x^{-1}, y] = [y, x]^{x^{-1}};
  \end{equation}
   \begin{equation}
 [x, y^{-1}] = [y, x]^{y^{-1}};
  \end{equation}
   \begin{equation}
 [x y, z] = [x, z]^{y} [y, z] = [x, z] [x, z, y]  [y, z];
  \end{equation}
  \begin{equation} \label{f1}
[x, yz] = [x, z] [x, y]^z=[x, z] [x, y] [x, y, z];
\end{equation}
\begin{equation} \label{f2}
[[x, y], z^x] [[z, x], y^z] [[y, z], x^y]  = 1.
\end{equation}
which hold for any elements $x, y, z$ of arbitrary group.

For a group $G$ define its transfinite lower central series
$$
G = \gamma_1 (G) \geq\gamma_2 (G) \geq \ldots \geq\gamma_{\omega} (G) \geq\gamma_{\omega+1} (G) \geq \ldots,
$$
where
$$
\gamma_{\alpha+1} (G) = \langle [g_{\alpha}, g] ~|~g_{\alpha} \in \gamma_{\alpha} (G), g \in G \rangle
$$
and if $\alpha$ is a limit ordinal, then
$$
\gamma_{\alpha} (G) = \bigcap_{\beta < \alpha}\gamma_{\beta} (G).
$$
The group $G$ is said to be {\it transfinitely nilpotent} if $\gamma_{\alpha} (G) =1$
for some ordinal $\alpha$, or simply {\it nilpotent} if $\alpha$ is a finite ordinal.
In particular, $G$ is residually nilpotent if and only if
$$
\gamma_{\omega}(G) = \bigcap_{i=1}^{\infty} \gamma_{i}(G) = 1.
$$
The smallest ordinal  $\alpha$ such that $\gamma_{\alpha} (G)
= \gamma_{\alpha+1} (G)$ is called the {\it length of the lower central series} of $G$.

Also, we can use the following lemma.

\begin{lemma}[Lemma 2.8.8 \cite{Kh}]\label{KhC}
If $A, B$ are normal subgroups of some group $G$, then the following inclusion holds
$$
 [A, \gamma_{k}(B)] \leq [A, \underbrace{B,\dots , B}_{k}], ~~~k\geq 1.
$$
\end{lemma}
The next inclusions follow from this lemma.
\begin{equation} \label{gsum}
[\gamma_{i}(G), \gamma_{j}(G)] \leq \gamma_{i+j}(G).
\end{equation}

\begin{equation} \label{iw}
 \gamma_{i}(\gamma_{\omega}(G)) \leq \gamma_{i\omega}(G).
\end{equation}

We say that $G$ is {\it virtually residually nilpotent},
if it contains a residually nilpotent subgroup of finite index.

The lower $p$-series of $G$ is defined as follows:
$$
\gamma^{(p)}_1(G)=G,~~  \gamma^{(p)}_2(G)=[G, \gamma^{(p)}_1(G)](\gamma^{(p)}_1(G))^p=[G,G]G^p,
$$
and recursively  $\gamma^{(p)}_{n+1}(G)=[G, \gamma^{(p)}_n(G)](\gamma^{(p)}_n(G))^p$
for every integer   $n \geq 1$.
Here $m$, $p \in \mathbb{N}$, $p\geq 2$, $(\gamma_{s}G)^{p^{k}}$
is the subgroup which is generated by
$p^{k}$-th powers of elements in   $\gamma_{s} G$.

It is easy to see that
$$
\gamma_{m+1}^{(p)} G= G^{p^m} (\gamma_{2}G)^{p^{m-1}}\cdot\ldots\cdot (\gamma_{m}G)^{p} \gamma_{m+1} G.
$$
Note that $\gamma_{m+1}^{(p)} G$ is a characteristic subgroup of  $G$ and
$$
[\gamma^{(p)}_k(G), \gamma^{(p)}_l(G)]\leq  \gamma^{(p)}_{k+l}(G).
$$
Group $G$ is said to be  {\it residually $p$-nilpotent}, if
$$
\gamma^{(p)}_{\omega}(G) = \bigcap_{i=1}^{\infty} \gamma^{(p)}_{i}(G) = 1.
$$
It is evident that finitely generated residually $p$-nilpotent group is residually $p$-finite.

\medskip

 Gruenberg \cite{G} called a class of groups  $\mathcal{P}$  by {\it root class}
if it satisfying the following three conditions:

(1) if a group lies in  $\mathcal{P}$, then so also is every subgroup;

(2) if $G$ and $H$ lie in  $\mathcal{P}$, then so also  $G \times H$;

(3) if $ G \geq H \geq K \geq 1$ is a series of subgroups,
each normal in its predecessor, and $G / H$, $H / K$, $K$ lie in  $\mathcal{P}$,
then $G$ is residually $\mathcal{P}$.

The following classes of groups are root classes:
solvabel groups, finite groups, class of finite $p$-groups for some  prime $p$.
Class of nilpotent groups fails to satisfy (3) and hence is not a root class.

For root classes  Gruenberg proved

\begin{prop} (\cite[Lemma 1.5]{G})
If $\mathcal{P}$ is a root class and $A$ is a normal subgroup
of $G$ such that $G / A$ lies in $\mathcal{P}$ and $A$ is residually $\mathcal{P}$,
then $G$ is residually $\mathcal{P}$.
\end{prop}

Also Gruenberg proved that if  $\mathcal{P}$ is a class of solvabel groups,
or finite groups or finite $p$-groups and $F$ is a free product of groups,
each of which is residually  $\mathcal{P}$, then $F$ itself is residually  $\mathcal{P}$.


Let $\Gamma$ be a group and $M$ is a left $\mathbb{Z}[\Gamma]$--module.
$M$ is called {\it residually nilpotent} if
$$
\bigcap_{k=1}^{\infty}  \Delta^k_{\Gamma}  M = 0,
$$
where $\Delta^k_{\Gamma}$,  is the $k$--th power of the augmentation ideal
$$
\Delta_{\Gamma} := \ker \{ \mathbb{Z}[\Gamma] \to \mathbb{Z} \}, \quad \Gamma \to 1.
$$

 In \cite{M} was proved

\begin{lemma}[Lemma 2, \cite{M}]
 Let $M$ be a finitely generated free abelian group of rank $n \geq 1$ and an
infinite cyclic group with generator $t$ acts on $M$ as a matrix $A \in GL_n(\mathbb{Z})$. Suppose
that the product of any collection of eigenvalues of $A - E$ is not $\pm 1$. Then the
module $M$ is residually nilpotent $\mathbb{Z}[\langle t \rangle]$--module.
\end{lemma}

\medskip

The residually nilpotence  groups of the form $\mathbb{Z}^n \rtimes_{\varphi} \mathbb{Z}$, $\varphi\in\mathrm{Aut}(\mathbb{Z}^{n})$,
are studied by  Aschen\-bren\-ner and   Friedl \cite{AF}.
They have found a criteria of residually nilpotence and residually $p$-finite for groups of this type.
If $P_{\varphi}(x)$ is the characteristic polynomial of the matrix $[\varphi]$,  $p_{i}(x)\in\mathbb{Z}[x]$,
$i=1,\dots, s,$ are its non-reducible factors, then the following proposition holds.

\begin{prop}\label{MFcri} \hspace{10ex}
\begin{itemize}
\item[a)]  $\mathbb{Z}^{n}\rtimes_{\varphi}\mathbb{Z}$ is residually nilpotent if and only if
 $p_{i}(1)\neq\pm1$, $i=1,\dots, s.$
\item[b)]  $\mathbb{Z}^{n}\rtimes_{\varphi}\mathbb{Z}$ is residually $p$-finite
if and only if  $p_{i}(1)\in p\mathbb{Z}$, $i=1,\dots, s.$
\end{itemize}
\end{prop}

Further we shall use the following claim.

\begin{lemma}\label{GpZ}
If a group $\mathbb{Z}^{n}\rtimes_{\varphi}\mathbb{Z}$ is residually  $p$-finite,
then for any non-trivial element $g\in\mathbb{Z}^{n}\rtimes_{\varphi}\mathbb{Z}$
there exists a  nilpotent group  $G_{p}\rtimes_{\widehat{\varphi}}\mathbb{Z}$,
where $G_{p}$ is a finite abelian  $p$-group, which is a quotient of  $\mathbb{Z}^{n}$,
$\widehat{\varphi}\in\mathrm{Aut}(G_{p})$ is induced by $\varphi$ and the image of $g$
under the homomorphism
$\mathbb{Z}^{n}\rtimes_{\varphi}\mathbb{Z} \to G_{p}\rtimes_{\widehat{\varphi}}\mathbb{Z}$
is non-trivial.
\end{lemma}

\begin{proof}
By assumption, for any non-trivial element $g\in\mathbb{Z}^{n}\rtimes_{\varphi}\mathbb{Z}$ there exists
a normal subgroup  $N_{g}\lhd\mathbb{Z}^{n}\rtimes_{\varphi}\mathbb{Z}$
of finite  $p$-index, such that $g\notin N_{g}$. It is clear, that  $N_{g}$ contains some
 $\gamma_{k+1}(\mathbb{Z}^{n}\rtimes_{\varphi}\mathbb{Z})$.
Let us consider the normal subgroup  $\widetilde{N}_{g}=N_{g}\bigcap\mathbb{Z}^{n}$.
It is evidently, that $g\notin\widetilde{N}_{g}$,
$\widetilde{N}_{g}\geq\gamma_{k+1}(\mathbb{Z}^{n}\rtimes_{\varphi}\mathbb{Z})$
and     $|\widetilde{N}_{g}:\mathbb{Z}^{n}|=p^{s}$ for same $s\in\mathbb{N}$.
Hence, the normal subgroup $\widetilde{N}_{g}$ is a kernel of the need
homomorphism onto $G_{p}\rtimes_{\widehat{\varphi}}\mathbb{Z}$.
Under this homomorphism the non-trivial element  $g\in\mathbb{Z}^{n}\rtimes_{\varphi}\mathbb{Z}$
goes to a non-trivial element. Here  $\widehat{\varphi}$ is the automorphism which
is induced by  ${\varphi}$ since
$\widetilde{N}_{g}$ is normal in  $\mathbb{Z}^{n}\rtimes_{\varphi}\mathbb{Z}$.
\end{proof}

\begin{lemma}\label{Pphi1} In the group $\mathbb{Z}^{n}\rtimes_{\varphi}\mathbb{Z}$
the following inclusions  hold
$$
\mathbb{Z}^{n}\geq\gamma_{k}(\mathbb{Z}^{n}\rtimes_{\varphi}\mathbb{Z})\geq P^{k-1}_{\varphi}(1)\mathbb{Z}^{n}, k\geq2.
$$
\end{lemma}

\begin{proof} In $\mathbb{Z}^{n}\rtimes_{\varphi}\mathbb{Z}$ holds
$$
\gamma_{k}(\mathbb{Z}^{n}\rtimes_{\varphi}\mathbb{Z})=([\varphi]-E)^{k-1}\mathbb{Z}^{n}\leq\mathbb{Z}^{n}, ~~k\geq2.
$$
The evident equalities
$$
P_{\varphi}(x)=P_{\varphi}(1)+(x-1)f(x), f(x)\in\mathbb{Z}[x],~~P_{\varphi}([\varphi])=0
$$ imply
$$P_{\varphi}(1)\mathbb{Z}^{n}=([\varphi]-E)f([\varphi])\mathbb{Z}^{n}\leq([\varphi]-E)\mathbb{Z}^{n}.$$
Hence,    we get
$$
\mathbb{Z}^{n}\geq\gamma_{k}(\mathbb{Z}^{n}\rtimes_{\varphi}\mathbb{Z})\geq P^{k-1}_{\varphi}(1)\mathbb{Z}^{n}, ~~k\geq2.
$$
\end{proof}

\section{Residually nilpotence of group extensions} \label{s3}

In this section we will use the following well known proposition (see, for example \cite{Kh})

\begin{prop}\label{TtoSec}
Let $G$ be a group, for $x \in G$ denote by $\overline{x}$
the image of $x$ in the quotient $G/G'$. Then for any natural number $n$ the map
$$
\theta : (\overline{x}_1 \otimes \overline{x}_2 \otimes \ldots \otimes \overline{x}_n)
\mapsto [x_1, x_2, \ldots, x_n] \gamma_{n+1} (G)
$$
induces a homomorphism
$$
\underbrace{G/G' \otimes G/G' \otimes \ldots \otimes G/G'}_n \to \gamma_{n} (G) / \gamma_{n+1} (G).
$$
\end{prop}

Let $G$ be an extension of a group  $A$ by a group
 $B$, i.e.  we have the short exact sequence:
$$
1 \to A \to G \to B \to 1.
$$
Denote $\overline{A}=A/\gamma_2(A)$ and consider the quotient
$\gamma_{k} (A) /\gamma_{k+1} (A)$. By Proposition \ref{TtoSec} this quotient
 is a homomorphic image of the tensor product:
$$
\overline{A}^{\otimes k}=\underbrace{\overline{A} \otimes \ldots \otimes \overline{A}}_k.
$$
The action of  $G$ on $\gamma_{k} (A) /\gamma_{k+1} (A)$ and on $\overline{A}^{\otimes k}$ by conjugations
 agreed with the homomorphism
$\overline{A}^{\otimes k} \rightarrow \gamma_{k} (A) /\gamma_{k+1} (A)$.

Conjugations by elements of $B$ on the quotient
$\overline{A}=A/\gamma_2 (A)$ induces  the linear transformation
$\beta :\overline{A} \rightarrow \overline{A}$:
$$
a^b\gamma_2 (A)=\beta(a)\gamma_2 (A),
$$
where  $a\in A$, $b\in B$. Also, the action by conjugations of $B$ on $\overline{A}$
can be extended to the action of the group ring   $\mathbb{Z}[B]$
on $\overline{A}$. Hence, the group $\overline{A}$ is a left  $\mathbb{Z}[B]$-module.
Denote by  $\Delta_B$ the augmentation ideal of the ring   $\mathbb{Z}[B]$.

\begin{theorem} \label{t3.1}
Let $G$ be an extension of  $A$ by  $B$ and

1) $A$ is residually nilpotent, i.e.  $\gamma_{\omega} (A)=1$;

2) $B$ is  nilpotent, i.e. $\gamma_s (B)=1$ for some $s\geq 2$;

3) there exists some natural number  $N\geq 1$ such that  $\Delta_B^N \overline{A}=0$.

Then  $G$ is residually nilpotent.
\end{theorem}

\begin{proof}
Using induction by $k\in \mathbb{N} \cup \left\{ 0 \right\}$ let us prove the inclusion
$$
\gamma_{s+N p(N,k)} (G) \subseteq \gamma_{k+1} (A),
$$
where $p(N,k)=1+N+\ldots +N^{k-1}$ for $k \geq 1$ and $p(N,0)=0$.

Since $\gamma_s (B)=1$, we have $\gamma_s (G) \subseteq A$.

Suppose that for  $k-1$ the inclusion holds.
The quotient  $\gamma_{k} (A) /\gamma_{k+1} (A)$ is the homomorphic image of the tensor product
$\overline{A}^{\otimes k}$.
It is enough  to prove that
$$
[\overline{A}^{\otimes k}, {}_{N^k}G]=0.
$$
Since the action of  $A$ on the quotient  $\overline{A}$ is trivial, we must prove that
$$
[\overline{A}^{\otimes k}, {}_{N^k}B]=0.
$$

Let  $a_1\otimes \ldots \otimes a_k$ be some element of $\overline{A}^{\otimes k}$ and
$b \in B$. Then
$$
[a_1\otimes \ldots \otimes a_k,b]=a_1^b\otimes \ldots \otimes a_k^b -a_1\otimes \ldots \otimes a_k=
$$
$$
=\beta a_1\otimes \ldots \otimes \beta a_k -a_1\otimes \ldots \otimes a_k=
$$
$$
=(\beta-id+id) a_1\otimes \ldots \otimes (\beta-id+id) a_k -a_1\otimes \ldots \otimes a_k=
$$
$$
=\left((\beta-id)\otimes \ldots \otimes (\beta-id)+ \ldots+  id\otimes \ldots \otimes (\beta-id) \right)
 a_1\otimes \ldots \otimes a_k.
$$
Here  $id$ is the identity transformation.

Hence, the following inclusion holds
$$
[\overline{A}^{\otimes k}, b]\subseteq T_k(\beta) \overline{A}^{\otimes k},
$$
where
$$
T_k(\beta) = (\beta-id)\otimes id \otimes \ldots \otimes id + \ldots+  id\otimes \ldots \otimes id\otimes (\beta-id)
$$
(the operator $\beta-id$ occurs only once in each term).

If $b_1,\ldots,b_m \in B$, then
$$
[\overline{A}^{\otimes k}, b_1,\ldots,b_m]\subseteq T_k(\beta_1)\cdot\ldots\cdot T_k(\beta_m)\overline{A}^{\otimes k}.
$$
For $m\geq N^k$ by condition 3) we have
$$
T_k(\beta_1)\cdot\ldots\cdot T_k(\beta_m)=0,
$$
and hence
$$
[\gamma_k A, {}_{N^k}G]\subseteq \gamma_{k+1} A.
$$
\end{proof}



\begin{theorem} \label{t3.2}
Let $G$ be an extension of  $A$ by $B$ and

1)  $A$ and  $B$ are residually nilpotent;

2) there are natural numbers  $N\geq 1$ and  $m\geq 2$ such that
$$
\Delta_B^N \overline{A} \subseteq m \overline{A};
$$

3) $ \bigcap\limits_{l\geq 1} m^l \overline{A}^{\otimes k}=0$, $k\in \mathbb{N}$.

Then $\gamma_{\omega^2} (G)=1$.
\end{theorem}

\begin{proof}
Using induction by  $k\in \mathbb{N}$, let us prove the inclusion
$$
\gamma_{(k+1)\omega} (G) \subseteq \gamma_{k} (A).
$$

From the equality  $\gamma_\omega (B)=1$ follows that $\gamma_\omega (G) \subseteq A$.

Suppose that for  $k-1$ the  inclusion holds.
The quotient  $\gamma_{k} (A) /\gamma_{k+1} (A)$ is a homomorphic image of the tensor product
$\overline{A}^{\otimes k}$.
It is enough to prove that
$$
 \bigcap\limits_{l\geq 1} [\overline{A}^{\otimes k}, {}_{l}G]\subseteq \gamma_{k} (A).
$$
Since  $A$ acts on the quotient  $\overline{A}$ as identity map, we have to  prove the inclusion
$$
 \bigcap\limits_{l\geq 1} [\overline{A}^{\otimes k}, {}_{l}B]\subseteq \gamma_{k} (A).
$$

Let  $a_1\otimes \ldots \otimes a_k$ be some element of  $\overline{A}^{\otimes k}$ and
$b \in B$. Then
$$
[a_1\otimes \ldots \otimes a_k,b]=a_1^b\otimes \ldots \otimes a_k^b -a_1\otimes \ldots \otimes a_k=
$$
$$
=\beta a_1\otimes \ldots \otimes \beta a_k -a_1\otimes \ldots \otimes a_k=
$$
$$
=(\beta-id+id) a_1\otimes \ldots \otimes (\beta-id+id) a_k -a_1\otimes \ldots \otimes a_k=
$$
$$
=\left((\beta-id)\otimes \ldots \otimes (\beta-id)+ \ldots+  id\otimes \ldots \otimes (\beta-id) \right)
 a_1\otimes \ldots \otimes a_k.
$$
Hence,
$$
[\overline{A}^{\otimes k}, b]\subseteq T_k(\beta) \overline{A}^{\otimes k},
$$
where
$$
T_k(\beta) = (\beta-id)\otimes id \otimes \ldots \otimes id+ \ldots+  id\otimes \ldots \otimes id\otimes (\beta-id)
$$
(the operator $\beta-id$ occurs only once in each term).

If $b_1,\ldots,b_s \in B$, then
$$
[\overline{A}^{\otimes k}, b_1,\ldots,b_s]\subseteq
T_k(\beta_1)\cdot\ldots\cdot T_k(\beta_s)\overline{A}^{\otimes k}.
$$
For $s\geq N^k$ by condition 3)
$$
T_k(\beta_1)\cdot\ldots\cdot T_k(\beta_s)\overline{A}^{\otimes k} \subseteq m \overline{A}^{\otimes k}
$$
and hence,
$$
[\gamma_k (A), {}_{N^k}G] \subseteq  (\gamma_k (A))^m \gamma_{k+1} (A),
$$
where $(\gamma_k (A))^m$ is the subgroup of $\gamma_k (A)$ that is generated by
$m$-th powers of commutators. We get
$$
 \bigcap\limits_{l\geq 1} [\gamma_k (A), {}_{lN^k}B]\subseteq
 \bigcap\limits_{l\geq 1} (\gamma_k (A))^{l m} \gamma_{k+1} (A)= \gamma_{k+1} (A).
$$
\end{proof}

\begin{theorem} \label{t3.3}
Let $G$ be an extension of  $A$ by $B$, $p \geq 2$  a prime number and

1)  $A$ is residually $p$-nilpotent;

2) $B$ is nilpotent, i.e. $\gamma_s (B)=1$ for some $s\geq 2$;

3) there are natural number  $N\geq 1$  such that
$$
\Delta_B^N \overline{A} \subseteq p \overline{A}.
$$
Then  $G$ is residually nilpotent.
\end{theorem}

\begin{proof}
Using induction by $k\in \mathbb{N}$ let us prove that for some natural number  $m_k$
holds
$$
\gamma_{m_k} (G) \subseteq \gamma_{k+1}^{(p)} (A).
$$

The equality  $\gamma_s (B)=1$ implies $\gamma_s (G) \subseteq A$.
Hence, we can put $m_0=s$.

Suppose that we have proved the inclusion
$$
\gamma_{m_{k-1}} (G) \subseteq \gamma_{k}^{(p)} (A).
$$
Let us consider the group $[\gamma_{k}^{(p)} (A), G]$.
Since
$$
\gamma_{k}^{(p)} (A)=
(\gamma_{1} (A))^{p^{k-1}}(\gamma_{2} (A))^{p^{k-2}}\cdot\ldots\cdot (\gamma_{k-1} (A))^{p} \gamma_{k} (A),
$$
and any subgroup  $(\gamma_{l} (A))^{p^{k-l}}$, $l=1,\ldots,k$, is characteristic in  $A$, then
$$
[\gamma_{k}^{(p)} (A), G]\subseteq \prod\limits_{l=1}^{k} [(\gamma_{l} (A))^{p^{k-l}},G].
$$
 Theorem \ref{t3.2} implies that there exists a natural number  $n_1\in \mathbb{N}$ such that
$$
[(\gamma_{1} (A))^{p^{k-1}},{}_{n_1}G] \subseteq (\gamma_{1} (A))^{p^{k}}\gamma_2 (A).
$$
Further, there exists  $n_2\in \mathbb{N}$ such that
$$
[\gamma_{2} (A), {}_{n_2}G] \subseteq (\gamma_{2} (A))^{p^{k-1}}\gamma_3 (A).
$$
Hence,
$$
[(\gamma_{1} (A))^{p^{k-1}},{}_{n_1+n_2}G] \subseteq (\gamma_{1} (A))^{p^{k}} (\gamma_{2} (A))^{p^{k-1}}\gamma_3 (A).
$$
We see that there exist  $n_1,n_2,\ldots,n_k \in \mathbb{N}$ such that
$$
[(\gamma_{1} (A))^{p^{k-1}},{}_{s_k}G] \subseteq \gamma_{k+1}^{(p)} (A),
$$
where $s_k=n_1+n_2+\ldots+n_k$.

Since
$$
[(\gamma_{2} (A))^{p^{k-2}},{}_{s_k-n_1}G] \subseteq \gamma_{k+1}^{(p)} (A),
$$
$$
[(\gamma_{3} (A))^{p^{k-3}},{}_{s_k-n_1-n_2}G] \subseteq \gamma_{k+1}^{(p)} (A),
$$
$$
....................................
$$
$$
[(\gamma_{k-1} (A))^{p},{}_{n_k}G] \subseteq \gamma_{k+1}^{(p)} (A),
$$
we have
$$
[\gamma_{k}^{(p)} (A),{}_{s_k}G] \subseteq \gamma_{k+1}^{(p)} (A).
$$
It means that for  $m_k=m_{k-1}+s_k$ the following inclusion holds:
$$
\gamma_{m_k} (G) \subseteq \gamma_{k+1}^{(p)} (A).
$$
Since
$$
\gamma_{\omega} (G) \subseteq \gamma_{\omega}^{(p)} (A)=1,
$$
 $G$ is residually nilpotent.
\end{proof}



\begin{example} \label{ExNres}
Let $\varepsilon_1,\ldots,\varepsilon_n=\pm 1$.
By Theorem \ref{t3.3} the following group
$$
G=\left\langle \, t, x_1,\ldots,x_n  \,  | \, t^{-1}x_it=x_i^{\varepsilon_i}, \,\,
               i=1,\ldots,n   \,   \right\rangle
$$
is residually nilpotent. Indeed, $G$ is a semi-direct product of free group
$F_n = \langle x_1,\ldots,x_n \rangle$ and infinite cyclic group
$\mathbb{Z} = \langle t \rangle$. Here $F_n$ is residually $p$-nilpotent,
$\mathbb{Z}$ is nilpotent. If we take $N=1$, $p=2$, then
for any $i = 1, 2, \ldots, n$,
$[x_i,t]=x_i^{-2} \,\, \mbox{or} \,\, [x_i,t]=1.$
It means that $\Delta_{\mathbb{Z}}(F_n^{ab}) \subseteq 2 F_n^{ab}$.
\end{example}


\section{Residually nilpotence of semi-direct products} \label{s4}

In this section we shall assume that $G$ is a semi-direct product,
i.e. we have a split short exact sequence:
$$
1 \to A \to G  \to B \to 1.
$$

It is  known  \cite{FR1} that if $A$ and $B$ are residually nilpotent groups,
$G = A \rtimes_{\varphi} B$, where the action of $B$ on $A$ is trivial
by modulo of the commutator subgroup $\gamma_2 (A)$, then
$$
\gamma_n (G) = \gamma_n (A) \gamma_n (B),~\mbox{for all}~ n \geq 1,
$$
and $G$ is residually nilpotent.

The main purpose of the present  section is to prove the following theorem.

\begin{theorem} \label{t5.1}
If $A$ is residually $p$-nilpotent for some prime number $p$, $B$
is residually nilpotent, and $G = A \rtimes_{\varphi} B$,
where the action of $B$ on $A$ is trivial by modulo  $\gamma_2^{(p)} (A)$, then
$$
\gamma_k(G) = \gamma^{(p)}_k(A) \gamma_k(B),~~k \geq 1,
$$
and $G$ is residually nilpotent.
\end{theorem}

To prove this theorem we will follow ideas from \cite{FR1}.
The proof of the following lemma is the same as proof of Lemma 3.2 from \cite{FR1}.

\begin{lemma}\label{l1}
1) Let $x \in \gamma^{(p)}_s(A)$, $y, w \in G$ and $[y, w] \in \gamma^{(p)}_q(A)$. Then
$$
[x, y] \in \gamma^{(p)}_{s+q}(A) \Leftrightarrow [x, y^w] \in \gamma^{(p)}_{s+q}(A).
$$

2) Let $x \in G$, $y, w \in G$ and $[x, y] \in \gamma^{(p)}_s(A)$, $[x, [y, w]] \in \gamma^{(p)}_{s+q}(A)$.
Then
$$
[x, y] \in \gamma^{(p)}_{s+q}(A) \Leftrightarrow [x, y^w] \in \gamma^{(p)}_{s+q}(A).
$$
\end{lemma}

Let us prove

\begin{lemma} \label{l2}
For any $n \geq 1$ holds
$$
[\gamma_n^{(p)}(A), \gamma_l(B)] \leq \gamma_{n+l}^{(p)}(A).
$$
\end{lemma}

\begin{proof}
Suppose that $l=1$ and using induction by $n$ we prove that
$$
[\gamma_n^{(p)}(A), B] \leq \gamma_{n+1}^{(p)}(A).
$$
If $n=1$ then $[A, B] \leq \gamma^{(p)}_2(A)$ that follows from assumption of theorem.
Let $n \geq 2$ and suppose that for all $r$ such that $1 \leq r \leq n-1$ holds
$$
[\gamma_r^{(p)}(A), B] \leq \gamma_{r+1}^{(p)}(A).
$$
Since
$$
[\gamma_n^{(p)}(A), B] = [(\gamma_{n-1}(A))^p [\gamma_{n-1}^{(p)}(A), A], B],
$$
it is enough to prove that
$$
[a_1^p [a_2, a], b] \in \gamma_{n+1}^{(p)}(A),~~~a_1, a_2 \in \gamma_{n-1}^{(p)}(A),~~a \in A, b\in B.
$$
Using the commutator identity (\ref{f1}), we get
$$
[a_1^p [a_2, a], b] = [a_1^p, b] \cdot [[a_1^p, b], [a_2, a]] \cdot [[a_2, a], b].
$$
Let us consider the commutator $[[a_2, a], b]$. By (\ref{f2})
$$
[a_2, a, b^{a_2}] [b, a_2, a^{b}] [a, b, a_2^{a}]  = 1.
$$
We have $a_2^a \in \gamma_{n-1}^{(p)}(A)$ and
$$
[a, b, a_2^{a}] \in [\gamma_{2}^{(p)}(A), \gamma_{n-1}^{(p)}(A)] \leq \gamma_{n+1}^{(p)}(A).
$$
Also by induction hypothesis
$$
[b, a_2] = [a_2, b]^{-1} \in \gamma_{n}^{(p)}(A).
$$
Hence
$$
[b, a_2, a] \in [\gamma_{n}^{(p)}(A), A] = \gamma_{n+1}^{(p)}(A).
$$
Since $[b, a_2] \in \gamma_{n}^{(p)}(A)$ and $a^b \in A$,
$$
[b, a_2, a^b] \in  \gamma_{n+1}^{(p)}(A).
$$
Then
$$
[a_2, a, b^{a_2}] \in  \gamma_{n+1}^{(p)}(A).
$$
Using  Lemma \ref{l1},  we get the need assertion.

Let us prove that
$$
[\gamma_n^{(p)}(A), \gamma_k(B)] \leq \gamma_{n+k}^{(p)}(A), ~~n, k \geq 1.
$$
Induction by $k$. We proved it for $k=1$. Suppose that
$$
[\gamma_n^{(p)}(A), \gamma_q(B)] \leq \gamma_{n+q}^{(p)}(A),
$$
for all $1 \leq q \leq k-1$ and $n \geq 1$.

Since
$$
[\gamma_n^{(p)}(A), \gamma_k(B)] = [\gamma_k(B), \gamma_n^{(p)}(A)] =
[[\gamma_{k-1}(B), B], \gamma_n^{(p)}(A)],
$$
it is enough to prove the inclusion
$$
[b_{k-1}, b, a] \in \gamma_{n+k}^{(p)}(A)
$$
for all $b_{k-1} \in \gamma_{k-1}(B)$, $b \in B$, and $a \in \gamma_n^{(p)}(A)$.
By the identity (\ref{f2})
$$
[b_{k-1}, b, a^{b_{k-1}}][a, b_{k-1}, b^a] [b, a, b_{k-1}^b]  = 1.
$$
By induction hypotheses
$$
[b, a, b_{k-1}^b] \in [\gamma_{n+1}^{(p)}(A), \gamma_{k-1}(B)] \leq \gamma_{n+k}^{(p)}(A).
$$
and
$$
[a, b_{k-1}, b] \in [\gamma_{n+k+1}^{(p)}(A), B] \leq \gamma_{n+k}^{(p)}(A).
$$
By Lemma \ref{l1}
$$
[a, b_{k-1}, b^a] \in \gamma_{n+k}^{(p)}(A)
$$
and we get
$$
[b_{k-1}, b, a^{b_{k-1}}] \in  \gamma_{n+k}^{(p)}(A).
$$
Also by Lemma \ref{l1}
$$
[b_{k-1}, b, a] \in  \gamma_{n+k}^{(p)}(A).
$$
\end{proof}

{\it Proof of Theorem \ref{t5.1}}
 It is enough  to prove the equality
\begin{equation} \label{eq}
\gamma_k(G) = \gamma^{(p)}_k(A) \gamma_k(B),~~k \geq 1.
\end{equation}
We use induction by $k$. For $k=1$
$$
G = A B \Leftrightarrow \gamma_1(G) = \gamma^{(p)}_1(A) \gamma_1(B),
$$
i.e. the need equality follows from the definition.

By assumption, $A$ has the $p$-central series
$$
A = \gamma_1^{(p)}(A) \geq \gamma_2^{(p)}(A) = A^p [A, A] \geq \ldots
$$
such that
$$
[\gamma_k^{(p)}(A), \gamma_l^{(p)}(A)] \leq \gamma_{k+l}^{(p)}(A).
$$
Also by assumption
$$
[A, B] \leq \gamma_2^{(p)}(A).
$$

Let us prove that
$$
\gamma_{n+1}(G) = \gamma^{(p)}_{n+1}(A) \gamma_{n+1}(B),
$$
under the  assumption that for all $1 \leq r \leq n$ holds
$$
\gamma_{r}(G) = \gamma^{(p)}_{r}(A) \gamma_{r}(B).
$$
Take a commutator
$$
[a_n b_n, a b],~\mbox{where}~a_n \in \gamma^{(p)}_{n}(A),~b_n \in \gamma_{n}(B),~a \in A, b \in B.
$$
By the commutator identities
$$
[a_n b_n, a b] = [a_n, b]^{b_n} \cdot [a_n, a]^{ b b_n} \cdot [b_n, b] \cdot [b_n, a]^b.
$$
We get
$$
[a_n b_n, a b] = [a_n, b]^{b_n} \cdot [a_n, a]^{ b b_n} \cdot [b_n, a]^{b [b_n, b]^{-1}}[b_n, b].
$$
The commutator $[b_n, b]$ lies in $\gamma_{n+1}(B)$.
Let us show that all other commutators lie in $\gamma^{(p)}_{n+1}(A)$.
Indeed, by Lemma \ref{l1}
$$
[a_n, b] \in \gamma^{(p)}_{n+1}(A)
$$
and thus $[a_n, b]^{b_n} \in \gamma^{(p)}_{n+1}(A)$. By the same reason
$$
[b_n, a]^{b [b_n, b]^{-1}} = ([a, b_n]^{-1})^{b [b_n, b]^{-1}} \in \gamma^{(p)}_{n+1}(A).
$$
Since $A$ is residually $p$-nilpotent, then
$$
[a_n, a]^{b b_n} \in \gamma^{(p)}_{n+1}(A).
$$
This completes the proof.



\bigskip

In \cite{G} for a semi-direct product $G = A \rtimes B$ was given the follow description of $\gamma_n(G)$.
Let $L_1 = A$, and if $n \geq 2$:
$$
L_n = \langle [A, \gamma_{n-1}(B)],~~[L_{n-1}, B], [A, L_{n-1}] \rangle.
$$
Then
$$
\gamma_n (G) = L_n  \rtimes \gamma_n (B).
$$

In particular,
$$
\gamma_2 (G) = \langle [A, B],  [A, A] \rangle  \rtimes \gamma_2 (B),
$$
$$
\gamma_3 (G) =
\langle [A, [B, B]], [\langle [A, B], [A, A] \rangle, B],~~[A, \langle [A, B], [A, A] \rangle]\rangle  \rtimes \gamma_3 (B).
$$

Let us show that for small $n$ it is possible to give simpler description of $\gamma_n(G)$. We shall use

\begin{lemma}
1) If  $H$ is a  normal subgroups of   $G$, then  $[A,H]$ is  normal in $G$;

2) The subgroup $[A, B]$ is a normal in $G$.
\end{lemma}

\begin{proof}
1) Is evident.

2) Follows from the commutator identities:
$$
[a,b]^{a_1}=[aa_1,b][a_1,b]^{-1},\quad [a,b]^{b_1}=[a^{b_1},b^{b_1}].
$$
\end{proof}



\begin{prop}
If $G = A \rtimes B$, then
$$
\gamma_2(G)=\gamma_2 (A) [A,B]  \rtimes \gamma_2 (B),
$$
$$
\gamma_3(G)=\gamma_3(A)   [A,B,A][B,B,A] [A,A,B][A,B,B] \rtimes \gamma_3 (B).
$$
\end{prop}

\begin{proof}
The first equality follows from the identity
$$
[a_1a_2,b_1b_2]=[a_1,a_2]^{b_2b_1} [a_1,b_2]^{b_1[a_1,a_2]^{b_2b_1}} [b_1,a_2]^{b_2[b_2,b_1]}  [b_1,b_2].
$$

Let us prove  the second equality.
At first,
$$
\gamma_3(G)=[\gamma_2(G),AB]=[\gamma_2(G),A][\gamma_2(G),B],
$$
since $[\gamma_2(G),A]\unlhd G$ and
$$
[w,ab]=[w,b] [w,a]^{b}=  [w,a]^{b[b,w]}[w,b]
$$
for every $w\in \gamma_2(G)$, $a\in A$, $b\in B$.

Secondly,
$$
[\gamma_2(G),A]=\gamma_3(A)[A,B,A][B,B,A],
$$
since
$$
[xyz,a]=[x,a]^{yz}[y,a]^{z}[z,a]
$$
and if $x\in \gamma_2(A)$, $y\in [A,B]$, $z\in [B,B]$, then
$$
[x,a]\in \gamma_3(A), \quad [y,a]\in [A,B,A], \quad [z,a]\in [B,B,A]
$$
and moreover $\gamma_3(A), [A,B,A] \unlhd G$.

Next,
$$
[\gamma_2(G),B]\subseteq \gamma_3(A) [A,A,B] [A,B,A][A,B,B]\gamma_3(B),
$$
since
$$
[xyz,b]=[x,b]^{yz}[y,b]^{z}[z,b]
$$
and if $x\in \gamma_2(A)$, $y\in [A,B]$, $z\in [B,B]$, then
$$
[x,b]^{yz}=[x,b]^{z}[[x,b]^{z},y^z] \in [A,A,B] [A,B,A],
$$
$$
[y,b]^{z}\in  [A,B,B], \quad [z,b]\in \gamma_3(B).
$$
Thus
$$
[xyz,b]\in [A,A,B] [A,B,A] [A,B,B] \gamma_3(B).
$$
Hence
$$
\gamma_3(G)=\gamma_3(A)   [A,B,A][B,B,A] [A,A,B] [A,B,A]  [A,B,B] \rtimes \gamma_3 (B).
$$
Since $[A,B,A]\unlhd G$, then
$$
\gamma_3(G)\subseteq \gamma_3(A)  [A,B,A][B,B,A] [A,A,B]  [A,B,B] \rtimes \gamma_3 (B).
$$
The opposite inclusion is evident.
\end{proof}

\begin{remark}
In general case subgroups  $[B,B,A]$, $[A,A,B]$, $[A,B,B]$, $\gamma_3 (B)$, are not normal in $G$.
\end{remark}

\medskip

In \cite{GM} was found  the lower central series of the pure braid group $P_2(\mathbb{K})$ of the  Klein bottle $\mathbb{K}$. Using our approach we can prove

 \begin{prop}[\cite{GM}]
$P_2 (\mathbb{K})$ is residually nilpotent.
 \end{prop}

  \begin{proof}


It is known  that $P_2 (\mathbb{K}) = \langle a_1, b_1, a_2, b_2 \rangle = F_2 \rtimes_{\varphi} s(P_1)$,
where $F_2 = \langle a_2, b_2 \rangle$ is the free group,
$$
s(P_1) = \langle a_1 a_2,  b_1 b_2~|~(b_2 b_1) (a_1 a_2) = (a_1 a_2)^{-1} (b_2 b_1) \rangle
$$
and the action of $s(P_1)$ on $F_2$ is defined by the rules
$$
a_1 a_2 :
\left\{%
\begin{array}{l}
  a_2 \to a_2, \\
  b_2 \to a_2^{-2} b_2, \\
\end{array}%
\right.~~~~
b_2 b_1 :
\left\{%
\begin{array}{l}
  a_2 \to a_2^{-1}, \\
  b_2 \to a_2 b_2 a_2. \\
\end{array}%
\right.~~~~
$$
Hence, $P_2 (\mathbb{K}) = F_2 \rtimes_{\varphi} F_2$.
It is easy to see that conjugations by $a_1 a_2$ and $b_2 b_1$
induce the following automorphisms of $F_2^{ab}$:
$$
[\overline{a_1 a_2}] =
\left(
\begin{array}{cc}
 1 & 0 \\
-2 & 1 \\
\end{array}%
\right),~~~
[\overline{b_2 b_1}] =
\left(
\begin{array}{cc}
 -1 & 0 \\
2 & 1 \\
\end{array}%
\right).
$$
Theorem \ref{t5.1} implies that $P_2 (\mathbb{K})$ is residually nilpotent.
  \end{proof}


\bigskip

\section{Free-by-cyclic groups $F_n \rtimes_{\varphi} \mathbb{Z}$}\label{FnZ}

In this section we consider the semi-direct product of the free group
$F_n = \langle x_1, x_2, \ldots, x_n \rangle$ and the infinite cyclic group
$\mathbb{Z} = \langle t \rangle$, where the conjugation by $t$ is induced
by automorphism $\varphi \in \mathrm{Aut}(F_n)$
$$
F_n \rtimes_{\varphi} \mathbb{Z} = \langle x_1, x_2, \ldots, x_n, t~|~t^{-1} x_i t = \varphi(x_i),~~i = 1, 2, \ldots, n \rangle.
$$

The automorphism $\varphi$ induces an automorphism of the
abelianization $F_n^{ab} = F_n / \gamma_{2}(F_n)$ that is free $\mathbb{Z}$-module.
We will denote this automorphism by $\overline{\varphi}$
and its matrix by $[\overline{\varphi}]$.
Denote by   $\overline{G}_k = (F^{ab}_{n})^{\otimes k}\rtimes_{\overline{\varphi}_{k}}\mathbb{Z}$,   $k\geq1$,
where  the automorphism  $\overline{\varphi}_{k}\in\mathrm{Aut}\left((F^{ab}_{n})^{\otimes k}\right)$
is induced by automorphism   $\varphi$. Hence, $\overline{\varphi}_{1} = \overline{\varphi}$.

Also, we will consider groups $\widehat{G}_k = \gamma_{k}(F_{n})/\gamma_{k+1}(F_{n})\rtimes_{\widehat{\varphi}_{k}}\mathbb{Z},$ $k\geq1,$
where $\widehat{\varphi}_{k}$ is the automorphism of  $\mathbb{Z}$-module
$\gamma_{k}(F_{n})/\gamma_{k+1}(F_{n})$ that is induced by the automorphism  $\varphi$.

For any $k\geq1,$ the pair of matrices  $[\overline{\varphi}_{k}]$ and $[\widehat{\varphi}_{k}]$ have the following property.

\begin{prop}\label{ChP}
Let $P_{\overline{\varphi}_{k}}(x)$ and  $P_{\widehat{\varphi}_{k}}(x)\in\mathbb{Z}[x]$
be characteristic polynomials of $[\overline{\varphi}_{k}]$ and $[\widehat{\varphi}_{k}]$,
respectively. Then the set of irreducible over  $\mathbb{Z}$  factors
of $P_{\widehat{\varphi}_{k}}(x)$ is a subset of the irreducible over
$\mathbb{Z}$ factors of $P_{\overline{\varphi}_{k}}(x)$.
\end{prop}
\begin{proof}
Let us prove that
for the automorphisms  $\overline{\varphi}_{k}$ and $\widehat{\varphi}_{k}$  and a polynomial  $f(x)\in\mathbb{Z}[x]$, the equality  $f(\overline{\varphi}_{k})=0$ implies  $f(\widehat{\varphi}_{k})=0$.
It is enough to prove that  for any vectors   $\overline{w}\in(F^{ab}_{n})^{\otimes k}$ and $\widehat{w}\in\gamma_{k}(F_{n})/\gamma_{k+1}(F_{n})$  the equality   $f(\overline{\varphi}_{k}) \overline{w}=0$ implies the equality $f(\widehat{\varphi}_{k}) \widehat{w}=0$.

By Proposition \ref{TtoSec} there exists the homomorphism of $\mathbb{Z}$-modules:
$$
\gamma : (F^{ab}_{n})^{\otimes k} \to \gamma_{k}F_{n}/\gamma_{k+1}F_{n},~~\overline{g}_{1}\otimes\dots\otimes\overline{g}_{k}\mapsto[ {g}_{1},\dots,
 {g}_{k}]\gamma_{k+1}F_{n},~~~\overline{g}_{i}\in F^{ab}_{n},~~ i=1,\dots k.
 $$
Let
$$
W=\langle\overline{u}\otimes\overline{u}\otimes\overline{v}_{1}\otimes\dots\otimes\overline{v}_{k-2}~|~ \overline{u}, \overline{v}_{i}\in F^{ab}_{n} \rangle_{\mathbb{Z}}<(F^{ab}_{n})^{\otimes k}.
$$ Then $W^{\gamma}=0\in\gamma_{k}(F_{n})/\gamma_{k+1}(F_{n})$ and $W^{\overline{\psi}}\leq W$, for any endomorphism    $\overline{\psi}\in\mathrm{End}\left((F^{ab}_{n})^{\otimes k}\right)$ which is induced by the endomorphism $\psi\in\mathrm{End}(F_{n})$.

If a non-trivial  element $\widehat{w}\in\gamma_{k}(F_{n})/\gamma_{k+1}(F_{n})$ is  a linear combination of the elementary commutators  $[x_{i_{1}},\dots, x_{i_{k}}]$, then some its  preimage  $\overline{w}$ under the action of  $\gamma$ is the similar linear combination of the elements  $\overline{x}_{i_{1}}\otimes\dots\otimes\overline{x}_{i_{k}}$ of  $(F^{ab}_{n})^{\otimes k}$, which does not lie in $W$. The endomorphisms  $f(\widehat{\varphi}_{k})$ and  $f(\overline{\varphi}_{k})$, which are induced by the  endomorphism  $f({\varphi})$, send  these linear combinations to linear combinations  which are equal under  $\gamma$ i.e. \ $(f(\overline{\varphi}_{k})\overline{w})^{\gamma}=f(\widehat{\varphi}_{k})\overline{w}^{\gamma}$. More exactly, if $f(\widehat{\varphi}_{k})\widehat{w}=\sum m_{i_{1},\dots,i_{k}}[x_{i_{1}},\dots, x_{i_{k}}]$,  then $f(\overline{\varphi}_{k})\overline{w}\in\sum m_{i_{1},\dots,i_{k}}\overline{x}_{i_{1}}\otimes\dots\otimes\overline{x}_{i_{k}}+W$. By assumption, the endomorphism  $f(\overline{\varphi}_{k})$ sends any non-trivial vector of   $(F^{ab}_{n})^{\otimes k}$ into $0$ \ i.e.
$\sum m_{i_{1},\dots,i_{k}}\overline{x}_{i_{1}}\otimes\dots\otimes\overline{x}_{i_{k}}\in W$, and so, the endomorphism  $f(\widehat{\varphi}_{k})$ sends any non-trivial element of the module  $\gamma_{k}(F_{n})/\gamma_{k+1}(F_{n})$ into $0$. Hence $f(\widehat{\varphi}_{k})=0$.

So if $P_{\overline{\varphi}_{k}}(x)$ is the characteristic polynomial of the matrix   $[\overline{\varphi}_{k}]$,
then $P_{\overline{\varphi}_{k}}(\widehat{\varphi}_{k})=0$. Hence, the minimal polynomial  $p_{\widehat{\varphi}_{k}}(x)$
of $[\widehat{\varphi}_{k}]$ divides
$P_{\overline{\varphi}_{k}}(x)$. It means, that any
irreducible factor of the characteristic polynomial of
$[\widehat{\varphi}_{k}]$ is an irreducible factor of the characteristic polynomial of $[\overline{\varphi}_{k}]$.
\end{proof}

Using Proposition  \ref{MFcri}, we get
\begin{cor}\label{GGcri}\hspace{1ex}For any  $k\geq1$ hold.
\begin{itemize}
\item[a)] If $\overline{G}_k$ is residually nilpotent then $\widehat{G}_k$ is residually nilpotent.
\item[b)] If $\overline{G}_k$ is residually $p$-finite then $\widehat{G}_k$ is residually $p$-finite.
\end{itemize}
\end{cor}

\subsection{Non-residually nilpotent groups}
In \cite{BMVW} was proved that if $G = F_n \rtimes_{\varphi} \mathbb{Z}$ and the matrix $[\overline{\varphi}] - E$
lies in $GL_n(\mathbb{Z})$, then $\gamma_{2}(G) =  \gamma_{3}(G)$ and
$$
\gamma_{\omega}(G) = \bigcap_{i=1}^{\infty} \gamma_i (G) = F_n.
$$
In particular, $G$ is not residually nilpotent.

The next proposition generalizes this result.

\begin{prop}\label{FnZn} Let $G = F_n \rtimes_{\varphi} \mathbb{Z}$ and $\overline{G} = F_n^{ab} \rtimes_{\overline{\varphi}} \mathbb{Z}$.
The following statements are equivalent
\begin{itemize}
\item[a)]   $\gamma_{\omega}(G) = F_n$.
\item[b)] $\gamma_{\omega}(\overline{G}) = F^{ab}_{n}$.
\item[c)] The matrix $[\overline{\varphi}] - E$
lies in $GL_n(\mathbb{Z})$ that is equivalent $\Delta_{\langle t \rangle} F^{ab}_{n} = F^{ab}_{n}$.
\end{itemize}

\end{prop}
\begin{proof} $a)\Rightarrow c)$
Suppose that  $\gamma_{\omega}(G) = F_{n}$, then $\gamma_{\omega}(\overline{G})= F^{ab}_{n}$.  Since
$$
F^{ab}_{n}\geq \gamma_{2}(\overline{G})\geq\dots \geq \gamma_{\omega}(\overline{G}),$$
we get
$$
F^{ab}_{n}=\gamma_{2}(\overline{G})=\dots=\gamma_{\omega}(\overline{G}).
$$
 Since in $\overline{G}$ holds
 $$
 ([\overline{\varphi}]-E)F_{n}^{ab}=\gamma_{2}(\overline{G}),
 $$ we have  $[\overline{\varphi}]-E\in\mathrm{GL}_{n}(\mathbb{Z})$.

$c)\Rightarrow a)$ Let $f \in F_{n}$. The map $f \mapsto[f, t],$  induces on the free   $\mathbb{Z}$-module $F_{n}^{ab}$ the linear transformation  $[\overline{\varphi}]-E$, which is an automorphism of $\mathbb{Z}$-module. So, there exist  elements $x_{1}^{k_{1i}}\cdots x_{n}^{k_{ni}},$    $i=1,\dots, n,$ such that  $[x_{1}^{k_{1i}}\cdots x_{n}^{k_{ni}}, t]=x_{i}w_{i}$, where $w_{i}\in\gamma_{2}(F_{n})$. Hence, $F_{n}\leq\gamma_{2}(G)$ and since $G/F_{n}=\mathbb{Z}$ is an abelian group,  we get $F_{n}=\gamma_{2}(G)$. Using induction one can see that  $x_{1},\dots, x_{n}\in\gamma_{s}(G), s\geq2,$ i.e. $F_{n}\leq\gamma_{\omega}(G)$.
Hence, we have the inclusions    $$F_{n}\leq\gamma_{\omega}(G)\leq\gamma_{2}(G)=F_{n},$$ from which follows the need equality   $F_{n}=\gamma_{\omega}(G)$.

$c)\Leftrightarrow b)$ Follows from the equality $$([\overline{\varphi}]-E)^{s}F_{n}^{ab}=\gamma_{s+1}(\overline{G}), s\geq1.$$
\end{proof}







\subsection{Groups with long lower central series}

Mikhailov  \cite{M} constructed a group with one defining relation that has the lower central series of length  $\omega^{2}$.
\begin{example}[\cite{M}]\label{CGM}
The group  $G_{M} = F_2 \rtimes_{\varphi} \mathbb{Z}$, where
$$
\varphi :
\left\{%
\begin{array}{l}
  a \to  b, \\
  b \to a b^3. \\
\end{array}%
\right.
$$
 has the lower central series of length $\omega^2$. At the same time,
$\overline{G}  = F_2^{ab} \rtimes_{\overline{\varphi}} \mathbb{Z}$
is residually nilpotent.
\end{example}

In \cite{M} for groups of the form
 $F_{n}\!\rtimes_{\varphi}\!\mathbb{Z}$ was formulated without prove the following theorem.

\begin{theorem}[\cite{M}] \label{Mres} Let $G = F_{n}\rtimes_{\varphi}\mathbb{Z}$.
If all groups  $\overline{G}_k = (F^{ab}_{n})^{\otimes k}\rtimes_{\overline{\varphi}_{k}}\mathbb{Z}$,   $k\geq1$,
are residually nilpotent, then $\gamma_{\omega^{2}}(G)=1$.
 Wherein if   $G$ is not residually nilpotent, then the length of its lower central series is equal to  $\omega^{2}$.
\end{theorem}

We give the full proof of this theorem. To do it, let us prove the following claim.

\begin{lemma}\label{Nilw2}
If all groups $\widehat{G}_i= \gamma_{i}(F_{n})/\gamma_{i+1}(F_{n})\rtimes_{\widehat{\varphi}_{i}}\mathbb{Z}$ are residually    nilpotent, then  $\gamma_{\omega^{2}}(G)=1$. Wherein if   $G$
is not residually nilpotent, then the length of its lower central series is equal to  $\omega^{2}$.
 \end{lemma}
\begin{proof}  It is well known that any quotient   $\gamma_{i}(F_{n})/\gamma_{i+1}(F_{n}),$ $i\geq1,$ is a free  $\mathbb{Z}$-module of finite rank. If we are considering an extension  $\mathbb{Z}^k \rtimes_{\varphi} \mathbb{Z}$ for some  $\varphi\in\mathrm{Aut}(\mathbb{Z}^k),$ then
$$
(\varphi-id)^{i}\mathbb{Z}^{k}=\gamma_{i+1}(\mathbb{Z}^{k} {\rtimes}_{\varphi} \mathbb{Z}), ~~~i\geq1.
$$
 Hence, by the assumption
 $$
 \bigcap_{m=1}^{\infty}(\widehat{\varphi}_{i}-id)^{m}\gamma_{i}(F_{n})/\gamma_{i+1}(F_{n})=0,~~i\geq1.
 $$
  So the following inclusion holds
$$
  \gamma_{i\omega}(F_{n}\rtimes_{\varphi}\mathbb{Z})\leq \gamma_{i+1}(F_{n}),~~i\geq1.
$$
Hence  $\gamma_{\omega^{2}}(G)=1$.

Further, if $\gamma_{\omega}(G) \not= 1$, then it is a free non-abelian group, which is normal in  $F_n$. So $\gamma_{i}(\gamma_{\omega}(G) ) \not= 1,$
for $i\ge1$. On the other side,  (\ref{iw}) implies $\gamma_{i}(\gamma_{\omega}(G) )\leq\gamma_{i\omega}(G), i\geq1$.
Using the fact that   $\gamma_{i}(\gamma_{\omega}(G) )$ is non-trivial for all $i \geq 1$, we get that  the length of the lower central series of
$G$ is equal to  $\omega^{2}$.
\end{proof}

Let us prove Theorem \ref{Mres}.
\begin{proof}\label{Tenz}
 Corollary \ref{GGcri} implies that from residual
 nilpotences  of groups
$\overline{G}_k$, $ k\geq1$,
follows the residual  nilpotences of groups
$\widehat{G}_k$, $k\geq1$.
By Lemma \ref{Nilw2} $\gamma_{\omega^{2}}(G)=1$, and if   $G$
is not residually nilpotent, then the length of its lower central series is equal to  $\omega^{2}$. This completes the proof of Theorem  \ref{Mres}.
\end{proof}

\section{Residually $p$-finiteness of groups $F_n \rtimes_{\varphi} \mathbb{Z}$}\label{pFnZ}

In this section  we give some sufficient conditions under which  $G = F_{n}\rtimes_{\varphi}\mathbb{Z}$ is  residually $p$-finite.

We shall use  the following statements.

\begin{lemma}\label{NGp}
Let  $G$ be   $l$-step nilpotent group.
If $[G, G]^{r}=1$ for some natural $r$, then $[G^{r^{l-1}}, G]=1$.
\end{lemma}

\begin{proof} If $l=1$, then $G$ is abelian and $[G, G] = 1$.

Let $\gamma_{l+1}(G)=1$,  by induction assumption  $[G^{r^{l-2}},G]\in\gamma_{l}(G)$, then from the commutator identity
$$
[ab,c]=[a,c][a,c,b][b,c]
$$ follows
$$
[G^{r^{l-1}},G]=[(G^{r^{l-2}})^{r},G]=[G^{r^{l-2}},G]^{r}=1.
$$
\end{proof}

\begin{lemma}\label{MN}
Let $M, N\leq\mathbb{Z}^{d}$ be submodules of finite  $p$-indexes,
 $\mathbf{v}_{1}, \mathbf{v}_{2},\dots, \mathbf{v}_{d}$ be some basis of  $N$.
If all non-trivial sums of the form  $\sum_{i=1}^{d}\alpha_{i}\mathbf{v}_{i}$,
where $\alpha_{i}\in\{0, 1,\dots, p-1\}$, do not lie in  $M$, then  $M<N$.
\end{lemma}

\begin{proof}
Since $M$ and $N$ are submodules of  $p$-indexes, there is a natural number
 $s$ such that $p^{s}N < M$. Moreover,  $M$  and  $N$ have consistent bases
$\mathbf{m}_{1},\dots, \mathbf{m}_{d}\in M$ and  $\mathbf{n}_{1},\dots, \mathbf{n}_{d}\in N$ such that
 $$
 p^{r_{i}}\mathbf{m}_{i}=p^{s}\mathbf{n}_{i},
$$
where $i=1,\dots, d,$ $0\leq r_{1}\leq\dots\leq r_{d}$.
Submodule  $N$ contains all  $\mathbf{m}_{i}$ for which  $r_{i}\leq s$, since
 $\mathbf{m}_{i}=p^{s-r_{i}}\mathbf{n}_{i}$. If $r_{i}>s$, then $p^{r_{i}-s}\mathbf{m}_{i}=\mathbf{n}_{i}\in M$.

By assumption any non-trivial sum of the form $\sum_{i=1}^{d}\alpha_{i}\mathbf{v}_{i}$,
where $\alpha_{i}\!\in\!\{0, 1,\dots, p-1\}$ does not  lie in  $M$. Hence, there is an epimorphism of the abelian group
 $(N+M)/M$ to the elementary abelian group
$\mathbb{Z}_{p}\oplus\dots\oplus\mathbb{Z}_{p}$ that is the direct sum of $d$ summands.
So, the rank of the  quotient  $(N+M)/M\leq\mathbb{Z}^{d}/M$ is equal  to $d$.
Hence,  the case  $r_{i}>s$,
i.e.  $p^{r_{i}-s}\mathbf{m}_{i}=\mathbf{n}_{i}\in M$ is not possible, since in this case  the  rank
of   $(N+M)/M$ less than $d$.
Hence, we get inclusion  $M<N$.
\end{proof}

Now we are ready to formulate the main result of the present section.

\begin{theorem}\label{Fpres}
Let $G = F_{n}\rtimes_{\varphi}\mathbb{Z}$.
If all groups $\widehat{G}_i = \gamma_{i}(F_{n})/\gamma_{i+1}(F_{n})\rtimes_{\widehat{\varphi}_{i}}\mathbb{Z}, i\geq1,$
are residually    $p$-finite, then  $G$ is residually    $p$-finite.
 \end{theorem}

\begin{proof}
Let $gt^{m}$ be some element of  $G$,
where $g\in F_{n}$ is non-trivial. Let us construct a finite $p$-group which is a homomorphic image of $G$ and
 the image of  $gt^{m}$ is non-trivial.

Since $F_{n}$ is residually  $p$-finite for any prime  $p$, it contains a normal subgroup of finite  $p$-index,
which does not contains  $g$. This subgroup contains a characteristic (verbal) subgroup $P$ such that the quotient
 $\mathbf{P} = F_n/P$
is also finite  $p$-group
 (\cite[Exercise 15.2.3]{KM}).
Then $\mathbf{P}\rtimes_{\widetilde{\varphi}}\mathbb{Z}$ is the quotient  of  $G$ by  $P$.
Here $\widetilde{\varphi}$ is the automorphism of the quotient
$\mathbf{P}$ which is induced by  $\varphi$. Let us show that
$\mathbf{P}\rtimes_{\widetilde{\varphi}}\mathbb{Z}$ is a nilpotent group.
Since $\gamma_{2}(\mathbf{P}\rtimes_{\widetilde{\varphi}}\mathbb{Z})\leq\mathbf{P}$
and $\mathbf{P}$ is nilpotent, it is enough to prove that  $[\mathbf{P}, t^{n_{1}} ,\dots, t^{n_{r}}]=1$ for some fixed
 $r$.

Find a minimal $s$ such that  $\gamma_{s+1}(F_{n}) \leq P<F_{n}$. For $1\leq i\leq s$ put
 $P_{i}=P\cap\gamma_{i}(F_{n})$. Then $\mathbf{P}_{i}=P_{i}/P_{i+1}$ is a submodule of finite  $p$-index of the free  $\mathbb{Z}$-module
$\gamma_{i}(F_{n})/\gamma_{i+1}(F_{n})$, and hence, is a free submodule of the same finite rank. For
 $g_{i} \in \gamma_{i}(F_{n})\backslash P$
denote by  $\overline{g}_{i}$  its image in the quotient $\gamma_{i}(F_{n})/\gamma_{i+1}(F_{n}).$

By assumption all groups  $\widehat{G}_i$
are residually $p$-finite. By Lemma  \ref{GpZ} they are residually  nilpotent groups of the form
 $G_{p}\rtimes_{\widehat{\varphi}}\mathbb{Z}$, where $G_{p}$ is a finite abelian  $p$-group.
Hence, in any   $\mathbb{Z}$-module  $\gamma_{i}(F_{n})/\gamma_{i+1}(F_{n})$
there exists a submodule  $\overline{M}_{i}$ of finite  $p$-index, which does not contains non-trivial  linear combinations of generators  of $\mathbf{P}_{i}$
with coefficients from the set
$\{ 0, 1, \ldots, p-1 \}$. By Lemma \ref{MN} we have inclusion
$\overline{M}_{i}<\mathbf{P}_{i}$. Since the quotients
$G_{p}\rtimes_{\widehat{\varphi}}\mathbb{Z}$ of groups
$\gamma_{i}(F_{n})/\gamma_{i+1}(F_{n})\rtimes_{\widehat{\varphi}_{i}}\mathbb{Z}$ are nilpotent, we get that for any  $\overline{g}_{i}$,
the commutator  $[\overline{g}_{i}, t^{n_{1}} ,\dots, t^{n_{r_{i}}}]\in\overline{M}_{i}<\mathbf{P}_{i}$
for some fixed  $r_{i}$.
Hence,
$$
[g_{i},t^{n_{1}} ,\dots, t^{n_{r_{i}}}]\in P_{i}\gamma_{i+1}(F_{n})\leq P\gamma_{i+1}(F_{n}),
$$
for any $g_{i}\in\gamma_{i}(F_{n})\backslash P$, and fixed  $r_{i}$.

Further, if  $[g_{i},t^{n_{1}} ,\dots, t^{n_{r_{i}}}]=pg_{i+1}$, where $p\in P$, then from the commutator identity
$$
[ab, c]=[a, c][a, c, b][b, c]
$$
and the fact that   $P$ is a characteristic subgroup, we get
$$
[pg_{i+1}, t^{n_{1}},\dots, t^{n_{r}}]\in P[g_{i+1}, t^{n_{1}},\dots, t^{n_{r}}].
$$
Since $[g_{i+1},t^{n_{1}} ,\dots, t^{n_{r_{i+1}}}]\in  P\gamma_{i+2}(F_{n}),$
we have
$$
[pg_{i+1}, t^{n_{1}} ,\dots, t^{n_{r_{i+1}}}]\in P\gamma_{i+2}(F_{n}).
$$

By induction we get that for any element
$g\in F_{n}\backslash P$ the commutator  $[g,t^{n_{1}} ,\dots, t^{n_{r}}]$, where $r=r_{1}+r_{2}+\dots+r_{s}$ lies in  $P\gamma_{s+1}(F_{n})=P\leqslant F_{n}$.
It means that $[\mathbf{P}, t^{n_{1}} ,\dots, t^{n_{r}}]=1$, i.e.  $\mathbf{P}\rtimes_{\widetilde{\varphi}}\mathbb{Z}$ is a nilpotent group.

Further, since  $p$-group $\mathbf{P}$ is finite, we can take a minimal natural number  $k_{0}$ such that  $\mathbf{P}^{p^{k_{0}}}=1$.
Moreover,  $\mathbf{P}$ contains the commutator subgroup of  $\mathbf{P}\rtimes_{\widetilde{\varphi}}\mathbb{Z}$.
By Lemma \ref{NGp}, using identity  \ $[t, \mathbf{P}]^{p^{k_{0}}}=1$ \ we get the identity
 \ $[t^{p^{k_{0}(l-1)}}, \mathbf{P}]=1$, \ where  $l$ is the  nilpotency class of
 $\mathbf{P}\rtimes_{\widetilde{\varphi}}\mathbb{Z}$. Hence,
$$
\mathbf{P}\rtimes_{\widetilde{\varphi}}\mathbb{Z}=
\left( \mathbf{P}\rtimes_{\widetilde{\varphi}}\mathbb{Z}_{p^{k_{0}(l-1)}}\right)\times\mathbb{Z}
$$
 and there exists a homomorphism of
 $\mathbf{P}\rtimes_{\widetilde{\varphi}}\mathbb{Z}$
to a finite  $p$-group $\mathbf{P}\rtimes_{\widetilde{\varphi}}\mathbb{Z}_{p^{k}}$,
where $k\geq k_{0}(l-1)$.

Finally, for any non-trivial  $g\in F_{n}$ we constructed a homomorphism of $G$  to a finite $p$-group    $\mathbf{P}\rtimes_{\widetilde{\varphi}}\mathbb{Z}_{p^{k}}$ such that the image of
 $g$ is non-trivial. It is evident that  the images of all elements  $gt^{m}$ for all $m$ are non-trivial.
On the other side, for any  element of the form  $t^{m}, m\in\mathbb{Z}\backslash\{0\}$ there exists a  homomorphism of
 $G$ to a finite
 $p$-group $\mathbb{Z}_{p^{k}}, p^{k}>m,$ such that the image of $t^{m}$ is non-trivial.
Hence, we proved that  $G$ is residually $p$-finite.
\end{proof}

As corollary we get

\begin{cor}\label{Npres}
Let $N_{n,c}$ be a free  nilpotent group of rank $n$ and class $c$, $G = N_{n,c}\rtimes_{\varphi}\mathbb{Z}$.
If all groups
$$
\gamma_{i}(N_{n,c})/\gamma_{i+1}(N_{n,c})\rtimes_{\widehat{\varphi}_{i}}\mathbb{Z}, ~~~1\leq i\leq c,
$$
are residually  $p$-finite, then  $G$ is residually  $p$-finite. Here  $\widehat{\varphi}_{i}$ is the automorphism of  $\mathbb{Z}$-module
$\gamma_{i}(N_{n,c})/\gamma_{i+1}(N_{n,c})$ that is induced by  $\varphi$.
\end{cor}

As was proved in \cite{M}, the residually nilpotence of $\overline{G}_k$, $k \geq 1$, does not imply the residually nilpotence  of $G = F_{n}\rtimes_{\varphi}\mathbb{Z}$.
Using  Theorem \ref{Fpres} it is possible to prove that the residually $p$-finiteness of $\overline{G}_k$, $k \geq 1$, implies  the residually $p$-finiteness  of
$G = F_{n}\rtimes_{\varphi}\mathbb{Z}$.

\begin{theorem}\label{Fdres}
Let $G = F_{n}\rtimes_{\varphi}\mathbb{Z}$. If all groups  $\overline{G}_{k}=(F^{ab}_{n})^{\otimes k}\rtimes_{\overline{\varphi}_{k}}\mathbb{Z}, \ k\geq1,$
are residually  $p$-finite, then
$G$ is residually  $p$-finite.
In particular, the length of the lower central series of $G$ is equal to  $\omega$.
 \end{theorem}

\begin{proof}
 Corollary \ref{GGcri} implies that from the residual
$p$-finiteness of groups
$\overline{G}_{k}, \ k\geq1,$
follows the residual  $p$-finiteness of all groups  \
$\widehat{G}_{k}=\gamma_{k}(F_{n})/\gamma_{k+1}(F_{n})\rtimes_{\widehat{\varphi}_{k}}\mathbb{Z},$ \ $k\geq1.$
By Theorem \ref{Fpres} the group $G$ is residually  $p$-finite. In particular, the length of its lower central series is equal to $\omega$.
\end{proof}

The following theorem gives a simple way to prove the residually nilpotence.
\begin{theorem}\label{Fdres1}
Let $G = F_{n}\rtimes_{\varphi}\mathbb{Z}$, $\varphi\in\mathrm{Aut}(F_{n})$ and $[\overline{\varphi}]\in\mathrm{GL}_{n}(\mathbb{Z})$. If all eigenvalues of  $[\overline{\varphi}]$ are integers, then  $G$ is residually nilpotent. Wherein,
 \begin{itemize}
\item[a)] if all the eigenvalues are equal to  $1$, then $G$ is residually  $p$-finite for any prime   $p$;
\item[b)] if there is at least one eigenvalue equal to  $-1$, then $G$ is residually $2$-finite.
\end{itemize}
 \end{theorem}
\begin{proof} a) By assumption, the characteristic polynomial of the automorphism $\overline{\varphi}$ is $P_{\overline{\varphi}}(x)=(x-1)^{n}$.
Since  $[\overline{\varphi}]\in\mathrm{GL}_{n}(\mathbb{Z})$, then
$[\overline{\varphi}]$ is conjugated to some unitriangular matrix $U\in\mathrm{GL}_{n}(\mathbb{Z})$. Hence, the tensor product
$[\overline{\varphi}_{k}]=[\overline{\varphi}]\otimes\dots\otimes[\overline{\varphi}]$
is conjugated  to the unitriangular matrix  $U\otimes\dots\otimes U\in\mathrm{GL}_{n^{k}}(\mathbb{Z})$. Hence,
$\overline{\varphi}_{k}\in\mathrm{Aut}\left((F^{ab}_{n})^{\otimes k}\right)$ has the characteristic polynomial
 $P_{\overline{\varphi}_{k}}(x)=(x-1)^{n^{k}}$, with one irreducible factor  $p_{1}(x)=x-1$.
Since $p_{1}(1)=0\in p\mathbb{Z}$ for any prime  $p$, then by Proposition \ref{MFcri} we get that any group
$\overline{G}_k, \ k\geq1,$ is residually  $p$-finite. By Theorem  \ref{Fdres} we get that  $G$ is residually
 $p$-finite for any  $p$.

 b) By assumption the characteristic polynomial of the automorphism $\overline{\varphi}$ is
$$
P_{\overline{\varphi}}(x)=(x-1)^{m_{1}}(x+1)^{m_{2}},~~m_{1}+m_{2}=n.
$$
 The matrix  $[\overline{\varphi}]\in\mathrm{GL}_{n}(\mathbb{Z})$  is conjugated  to some   triangular matrix
$T\in\mathrm{GL}_{n}(\mathbb{Z})$ in which $m_{1}$ elements on the diagonal  are equal to  $1$  and $m_{2}$ elements on the diagonal are equal to $-1$. Hence, the tensor product
$[\overline{\varphi}_{k}]=[\overline{\varphi}]\otimes\dots\otimes[\overline{\varphi}]$
is conjugated to the triangular matrix $T\otimes\dots\otimes T\in\mathrm{GL}_{n^{k}}(\mathbb{Z})$ with diagonal elements
 $1$ and  $-1$. Hence, the characteristic polynomial  of
$\overline{\varphi}_{k},$
is $P_{\overline{\varphi}_{k}}(x)=(x-1)^{m_{1k}}(x+1)^{m_{2k}}$,
where $m_{1k}m_{2k}\neq0$ and $m_{1k}+m_{2k}=n^k$.
This characteristic polynomial has two irreducible factors $p_{1}(x)=x-1$ and $p_{2}(x)=x+1$. Since  $p_{1}(1)=0$ and  $p_{2}(1)=2$, then
 $p_{1}(1), p_{2}(1)\in2\mathbb{Z}$. By Proposition \ref{MFcri} any group  $\overline{G}_{k}$, $k\geq1,$
is residually $2$-finite. Theorem  \ref{Fdres} implies that $G$ is residually  $2$-finite.
\end{proof}

\section{Groups $F_2 \rtimes_{\varphi} \mathbb{Z}$} \label{s8}

In this section we are considering groups  $G = F_2 {\rtimes}_\varphi \mathbb{Z}$ that are  semi-direct product of free group  $F_2=\left\langle x,y \right\rangle$ of rank 2
and the infinite cyclic group  $\mathbb{Z}=\left\langle t \right\rangle$:
$$
G=\left\langle \, x, y, t  \, | \, t^{-1}xt=\varphi(x),\,\, t^{-1}yt=\varphi(y)   \,   \right\rangle ,
$$
where $\varphi$ is an automorphism of  $F_2$. The automorphism $\varphi$ induces the automorphism  $\overline{\varphi}$
of the abelianization  $F^{ab}_{2}=F_2/\gamma_2 (F_2)$. In the present section, using results of Sections  \ref{FnZ} and \ref{pFnZ} we prove the following
criteria of the residually nilpotence.
\begin{theorem}\label{NAF2Z}
The group  $G = F_2 {\rtimes}_\varphi \mathbb{Z}$  is residually nilpotent if and only if  $\mathrm{det}
[\overline{\varphi}]=1$ and $\mathrm{tr}[\overline{\varphi}]\not\in \{1, 3 \}$, or $\mathrm{det}
[\overline{\varphi}]=-1$ and $\mathrm{tr}[\overline{\varphi}]\equiv0~ (\mathrm{mod}~ 2)$. At the same time, if $\mathrm{det}
[\overline{\varphi}]=1$, then $G$ is residually $p$-finite for any prime divisor of $\mathrm{tr}[\overline{\varphi}]-2$, and if  $\mathrm{det}[\overline{\varphi}]=-1$, then  $G$ is residually  $2$-finite.
\end{theorem}

Also, for any group of this type we find the length of its lower central series.
\begin{theorem}\label{F2Zc}
For the lower central series of $G = F_2 {\rtimes}_\varphi \mathbb{Z}$  there exist only three possibilities:
\begin{itemize}
\item[a)]  $\gamma_{\omega}(G)=\gamma_{2}(G)$;
\item[b)]  $\gamma_{\omega}(G)=1$;
\item[c)]  $\gamma_{\omega^{2}}(G)=1$ and  the length of the lower central series is equal to $\omega^{2}$.
\end{itemize}
\end{theorem}

\begin{problem}
Is this theorem true if we take $F_n$, $n>2$, instead $F_2$?
\end{problem}


\subsection{Groups with the  lower central series of length $2$} At first, let us describe groups for which  the length of the lower central series is equal $2$.

\begin{prop}\label{F2Znr} For the group  $G = F_2 {\rtimes}_\varphi \mathbb{Z}$ holds $\gamma_{\omega}(G) = F_{2}$ if and only if $\mathrm{det}[\overline{\varphi}]=1,$  $\mathrm{tr}[\overline{\varphi}] \in \{1, 3\},$ or $\mathrm{det}[\overline{\varphi}]=-1,$  $\mathrm{tr}[\overline{\varphi}]=\pm1$.
\end{prop}
\begin{proof}  By Proposition \ref{FnZn} the equality  $\gamma_{\omega}(G) = F_{2}$ is equivalent to the condition  $[\overline{\varphi}]-E\in\mathrm{GL}_{2}(\mathbb{Z})$. This  is equivalent to the fact that the characteristic polynomial $P_{[\overline{\varphi}]}(x)=x^{2}-\mathrm{tr}[\overline{\varphi}]x+\mathrm{det}[\overline{\varphi}]$ for $x=1$ is equal to $\pm1$. Hence, if  $\mathrm{det}[\overline{\varphi}]=1$, then $1-\mathrm{tr}[\overline{\varphi}]+1=\pm1$; if $\mathrm{det}[\overline{\varphi}]=-1$, then $1-\mathrm{tr}[\overline{\varphi}]-1=\pm1$. These give the need values. If $\mathrm{det}[\overline{\varphi}]=1$, then $\mathrm{tr}[\overline{\varphi}] \in \{1, 3\}$. If $\mathrm{det}[\overline{\varphi}]=-1$, then $\mathrm{tr}[\overline{\varphi}]=\pm1$.
\end{proof}

From this Proposition followings that the braid group $B_3$ has the short lower central series. This fact was proved by Gorin and Lin \cite{GL}.
\begin{example}[\cite{GL}]
The braid group  on 3 strands $B_3 = \langle \sigma_1,  \sigma_2  ~|~\sigma_1 \sigma_2 \sigma_1 =   \sigma_2 \sigma_1   \sigma_2 \rangle$  is the semi-direct product
$$
B_3= F_2 {\rtimes}_\varphi \mathbb{Z},
$$
where
 $F_2=\left\langle \, u=\sigma_2 \sigma_1^{-1}, v= \sigma_1\sigma_2 \sigma_1^{-2} \,\right\rangle = B_3'$
and
$\mathbb{Z}=\left\langle \, \sigma_1 \,\right\rangle$
acts on $F_2$ by the formulas
$$
\varphi(u)=u^{\sigma_1}=uv^{-1},\quad \varphi(v)=v^{\sigma_1}=u.
$$
We see that $\mathrm{det}[\overline{\varphi}]=1,$  $\mathrm{tr}[\overline{\varphi}]=1$ and  by Proposition \ref{F2Znr}$, \gamma_{2}B_{3}=\gamma_{\omega}B_{3}=F_2$.
\end{example}

\subsection{Groups with the  lower central series of  length $\omega$}
Note that if $G=F_2 {\rtimes}_\varphi \mathbb{Z}$ is residually nilpotent, then $\gamma_{\omega}(G) = 1$. On the other side, $G$ contains $F_2$ and so $\gamma_{i}(G)$ is non-trivial for all $i \geq 1$, hence, in this case $G$ has the length of the lower central series equal to $\omega$.

The eigenvalues of any matrix  in $\mathrm{GL}_{2}(\mathbb{Z})$ either both lie in $\mathbb{Z}$ and are equal to  $\pm1$, or both do not lie in  $\mathbb{Z}$.

Hence, the set of all groups of the form $F_2 {\rtimes}_\varphi \mathbb{Z}$  is the disjoint union of two subsets. For groups from the first subset Theorem \ref{Fdres1} implies

\begin{prop}\label{F2Zzvr} Let $G=F_2 {\rtimes}_\varphi \mathbb{Z}$ and  $[\overline{\varphi}]$ has only integer eigenvalues. Then $G$ is residually nilpotent and in this case
 \begin{itemize}
\item[a)] if both eigenvalues are equal to  $1$, then $G$ is residually  $p$-finite for any prime   $p$;
\item[b)] if  at least one eigenvalue is equal to  $-1$, then $G$ is residually $2$-finite.
\end{itemize}
\end{prop}

For groups from the second subset (without integer eigenvalues) holds

\begin{prop}\label{F2Zr} Let  $G = F_2 {\rtimes}_\varphi \mathbb{Z}$ the determinant $det [\overline{\varphi}] = 1$ and the eigenvalues of $[\overline{\varphi}] $ are not integers. If $\mathrm{tr}[\overline{\varphi}]\not\in \{ 1, 3\}$, then $G$ is residually $p$-finite for any prime   $p$ which divides  $\mathrm{tr}[\overline{\varphi}]-2$.
\end{prop}
\begin{proof} The characteristic polynomial of $[\overline{\varphi}]$ is the irreducible polynomial
$f(x)=x^{2}-\mathrm{tr}[\overline{\varphi}]x+1 \in\mathbb{Z}[x].$
By the assumption $\mathrm{tr}[\overline{\varphi}]-2 = -f(1)=pl$ for some integer $l$ and so
$$
f(x)=x^{2}-(2+pl)x+1 \in\mathbb{Z}[x].
$$
The matrix $[\overline{\varphi}]$  has two different eigenvalues $r, r^{-1} \in \mathbb{C}$ and is conjugated  to the diagonal matrix $\mathrm{diag}(r, r^{-1})$, where $r+r^{-1} = 2+pl \in 2+p\mathbb{Z}$. Since $2+pl=\mathrm{tr}[\overline{\varphi}] \not\in \{1, 3\},$  then $p \not\in \{\pm1\}$.

From the equality
$$
(r^{n-1}+r^{-(n-1)})(r+r^{-1})=r^{n}+r^{-n}+r^{n-2}+r^{-(n-2)}\in4+p\mathbb{Z}
$$
 follows that all pairs  $r^{n},r^{-n}, n\in\mathbb{Z}\backslash\{0\},$ are roots of irreducible over $\mathbb{Z}$  polynomials of the form
 $$
 f(x)=x^{2}-(2+pl)x+1, ~~l\in\mathbb{Z}\backslash\{0\}.
 $$
Using induction by  $k\geq1$ it is not difficult to prove that   $[\overline{\varphi}_k]=[\overline{\varphi}]^{\otimes k}$ is conjugated to the diagonal matrix  $(\mathrm{diag}(r, r^{-1}))^{\otimes k}$ in which the diagonal elements  $r^{n},r^{-n}, n\in\mathbb{Z},$ arrive in pairs. Hence, all irreducible factors of the characteristic polynomial of this matrix have the form
$$
f(x)=x^{2}-(2+pl)x+1,~~l\in\mathbb{Z}\backslash\{0\}~\mbox{or}~f(x)=x-1.
$$ So $f(1)\in p\mathbb{Z}$. By Proposition \ref{MFcri} all groups
 $(F^{ab})^{\otimes k}{\rtimes}_{\overline{\varphi}_k} \mathbb{Z}$ are residually $p$-finite for any prime  $p$ that is a divisor of $\mathrm{tr}[\overline{\varphi}]-2$.  Theorem \ref{Fdres} implies that  $G$ is residually $p$-finite for the same  prime  $p$.
\end{proof}

\begin{cor}
If the matrix $[\overline{\varphi}]$ lies in the congruence-subgroup  $\Gamma_2(m)\leq \mathrm{GL}_2(\mathbb{Z})$, then
$F_2 {\rtimes}_\varphi \mathbb{Z}$  is residually $p$-finite for any prime  $p$ which divides  $m$.
\end{cor}

From the previous results follows criterion of residually nilpotence of  $G = F_2 {\rtimes}_\varphi \mathbb{Z}$ with $\det[\overline{\varphi}] = 1$.
\begin{theorem}\label{F2Zr1} Suppose that for the group $G=F_2 {\rtimes}_\varphi \mathbb{Z}$ the determinant $\det[\overline{\varphi}] = 1$. Then  $G$  is residually nilpotent if and only if  $\mathrm{tr}[\overline{\varphi}]\not\in \{1, 3 \}$.
\end{theorem}

\begin{proof} As we noted above, both eigenvalues of  $[\overline{\varphi}] \in\mathrm{GL}_{2}(\mathbb{Z})$ either  lie in  $\mathbb{Z}$ and are equal to $\pm1$, or do not lie in  $\mathbb{Z}$.

Let $\mathrm{tr}[\overline{\varphi}]\not\in \{1, 3\}$. If the eigenvalues of  $[\overline{\varphi}]$   are integers, then by Lemma  \ref{F2Zzvr} the group $G$ is residually nilpotent.
If the eigenvalues of  $[\overline{\varphi}]$   are not  integers, then by Lemma \ref{F2Zr} the group  $G$ is residually nilpotent.

If $\mathrm{tr}[\overline{\varphi}] \in \{1, 3\}$, then the need assertion follows from Lemma  \ref{F2Znr}.
\end{proof}

\medskip

For the studying of the case $\det[\overline{\varphi}] = -1$, recall the definition.
A group is said to be {\it virtually residually nilpotent} if it contains a subgroup of finite index which is residually nilpotent. We know that not any group of the form $F_{n} {\rtimes}_\varphi \mathbb{Z}$ is residually nilpotent. On the other side, Azarov \cite{Azar}  proved that any semi-direct product of a  finitely generated  residually $p$-finite group by a residually $p$-finite group is virtually residually $p$-finite and hence, contains residually nilpotent subgroups of finite index. From this result follows

\begin{prop} \label{vrn}
Any group   $F_{n} {\rtimes}_\varphi \mathbb{Z}$ is virtually residually nilpotent.
\end{prop}

Further  we find a subgroup of finite index in $F_{2} {\rtimes}_\varphi \mathbb{Z}$ which is residually nilpotent.

\begin{lemma}\label{F2Zvr} The group $G= F_2 {\rtimes}_\varphi \mathbb{Z}$ contains a residually nilpotent subgroup of index  $2$ that is isomorphic to  $F_{2}{\rtimes}_{\varphi^{2}}\mathbb{Z}$ in all cases except $\mathrm{det}[\overline{\varphi}]\!=\!-1\!,$  $\mathrm{tr}[\overline{\varphi}]=\pm1$. In the cases $\mathrm{det}[\overline{\varphi}]=-1,$ $\mathrm{tr}[\overline{\varphi}]=\pm1$ it contains a residually nilpotent subgroup of index  $4$ that is isomorphic to     $F_2 {\rtimes}_{\varphi^{4}} \mathbb{Z}$.
\end{lemma}

\begin{proof}It is well known that any matrix  $M\in\mathrm{GL}_{2}(\mathbb{Z})$  satisfies to the charac\-teris\-tic equation
$$
x^{2}-\mathrm{tr}M \cdot x+\mathrm{det}M = 0.
$$
 So $M^{2}=(\mathrm{tr}M)M\pm E$ and  $\mathrm{tr}(M^{2})=(\mathrm{tr}M)^{2}-2$ for $\mathrm{det}M=1$, and  $\mathrm{tr}(M^{2})=(\mathrm{tr}M)^{2}+2$ for $\mathrm{det}M=-1$.

In our group
$$
G = F_2 {\rtimes}_\varphi \mathbb{Z}=\langle x, y, t ~| ~t^{-1}xt=\varphi(x), ~~t^{-1}yt=\varphi(y)\rangle,
$$
the subgroup that is generated by elements $x, y, t^{2}$ has index  $2$ and is isomorphic to  $F_2 {\rtimes}_{\varphi^{2}}\mathbb{Z}$.
If $\mathrm{det}[\overline{\varphi}]\neq-1$ or $\mathrm{tr}[\overline{\varphi}]\neq\pm1,$ then $\mathrm{tr}([\overline{\varphi}]^{2})=(\mathrm{tr}[\overline{\varphi}])^{2}\pm2\not\in \{1, 3\}$. So  by Theorem $\ref{F2Zr1}$ the group $F_2 {\rtimes}_{\varphi^{2}}\mathbb{Z}$ is residually nilpotent.

In the case  $\mathrm{det}[\overline{\varphi}]=-1,$ $\mathrm{tr}[\overline{\varphi}]=\pm1$, the subgroup that is generated by  $x, y, t^{4}$ has index  $4$, and is isomorphic to $F_2 {\rtimes}_{\varphi^{4}}\mathbb{Z}$. Since $[\overline{\varphi^{4}}]=[\overline{\varphi}]^{4}$ and $\mathrm{tr}([\overline{\varphi}]^{2})=(\mathrm{tr}[\overline{\varphi}])^{2}+2=3$, in this case  $\mathrm{det}([\overline{\varphi}]^{4})=1$ and $\mathrm{tr}([\overline{\varphi}]^{4})=3^{2}-2=7$. By Theorem  $\ref{F2Zr1}$  $F_2 {\rtimes}_{\varphi^{4}}\mathbb{Z}$ is residually nilpotent.
\end{proof}

\begin{cor}\label{GM} The Mikhailov group  $G_{M}=F_{2}\rtimes_{\varphi}\mathbb{Z}$,
where  $\varphi(x)=y,$ $\varphi(y)=xy^{3}$, 
contains residually nilpotent group of index  $2$, which is generated by  $x, y, t^{2}$. This subgroup  is isomorphic to  $F_{2}\rtimes_{\varphi^{2}}\mathbb{Z}$, 
and   is residually $3$-finite.
\end{cor}

Now we are ready to prove conditions under which groups with $\mathrm{det}[\overline{\varphi}]=-1$ are residually nilpotent.

\begin{theorem}\label{F2Z2r} If $\mathrm{det}
[\overline{\varphi}]=-1$ and  $\mathrm{tr}[\overline{\varphi}]\equiv0(\mathrm{mod}2)$, then $G = F_2 {\rtimes}_\varphi \mathbb{Z}$ is residually $2$-finite.
\end{theorem}
 \begin{proof} By Lemmas  \ref{F2Zvr} and \ref{F2Z2r}, $G$ contains a residually nilpotent subgroup $H = \langle x, y, t^{2} \rangle$ of index  $2$. Since $[\overline{\varphi^{2}}]=[\overline{\varphi}]^{2},$ $\mathrm{det}
([\overline{\varphi}]^{2})=1$ and $\mathrm{tr}([\overline{\varphi}]^{2})=(\mathrm{tr}[\overline{\varphi}])^{2}+2$, by Lemma \ref{F2Zr} $H$  is residually $2$-finite. Hence, $H$ has decreasing central series of normal subgroups of finite $2$-indexes with trivial intersection:
$$
H \geq H_1 \geq H_2 \geq \ldots,~~\bigcap_{i=1}^{\infty} H_i = 1.
$$
Since $H$ is finitely generated and $|H : H_i| = 2^{n_i}$ for some $n_i \in \mathbb{N}$, by (\cite[Exercise 15.2.3]{KM}) $H_i$ contains some verbal subgroup $V_i$ of $H$ which has finite $2$-index in $H$. It is easy to see that $H$ is normal in $G$. Hence, we have
$$
G \vartriangleright H \vartriangleright H_i \vartriangleright V_i,~~~i \geq 1.
$$
Using the fact that $V_i$ is a verbal subgroup of $H$ and $H \lhd G$, we get that $V_i$ is normal in $G$ and  $|G/V_i| = 2^{n_i+1}$. Further,  the series of groups
$$
W_k = \bigcap_{i=1}^{k} V_i
$$
is a normal series with trivial intersection and $G / W_k$ are finite 2-groups.
Hence, $G$ is residually  $2$-finite.
\end{proof}

\subsection{Groups with the length of the lower central series  $\omega^2$}

In this subsection we prove that in all remaining cases
 for $\mathrm{det}[\overline{\varphi}]$ and $\mathrm{tr}[\overline{\varphi}]$, group
$F_{2}\rtimes_{\varphi}\mathbb{Z}$ has the length of the lower central series equal to $\omega^2$.

\begin{theorem}\label{w2}
Let $G = F_{2}\rtimes_{\varphi}\mathbb{Z}$. Suppose that  $\mathrm{det}
[\overline{\varphi}]=-1$ and  $\mathrm{tr}[\overline{\varphi}]$ is an odd number,  $\mathrm{tr}[\overline{\varphi}]\neq\pm1$, then $\gamma_{\omega^{2}}(G)=1$.
\end{theorem}
\begin{proof} In these cases the characteristic polynomial of  $[\overline{\varphi} ]$ has the real roots  $r$ and $-r^{-1}$,
such that $r-r^{-1}\in\mathbb{Z}\backslash\{0, \pm1\}$. Also, $[\overline{\varphi}]$ is conjugated to the diagonal matrix  $\mathrm{diag}(r, -r^{-1})$.

Using induction by   $k$, it is possible to prove that $[\overline{\varphi}_{k}]=[\overline{\varphi}]^{\otimes k}$ is conjugated to the diagonal matrix
with diagonal elements  $\pm r^{i}, i\in\mathbb{Z}$. If $k=2n+1, n\geq0,$ then  all $i$ are equal to all odd numbers such that $|i| \leq k$. The set of
 diagonal elements is divided on the pairs  $\{r^{i}, -r^{-i}\}$ or $\{-r^{i}, r^{-i}\}$.
 If $k=2n, n\geq1,$ then all $i$ are equal to all even numbers such that $|i| \leq k$.
In this case the set of diagonal elements is  divided on the pairs of the form $\{r^{i}, r^{-i}\}$ or $\{-r^{i}, -r^{-i}\}$.
 Furthermore, in the diagonal matrix which is conjugated to  $[\overline{\varphi}]^{\otimes k}$ there exists only one pair $\{r^k, -r^{-k} \}$ for $k=2n+1$ and only one pair $\{r^k, r^{-k} \}$ for $k=2n$.

It is evident  that our pairs of roots  for  $i\neq0$ are the roots  of polynomials of degree 2:
$$
f_{\rho}(x)=x^{2}+\rho x-1~\mbox{or}~ f_{\delta}(x)=x^{2}+\delta x+1,
$$
where $\rho=\pm(r^{i}-r^{-i}),$ $i=2n+1, n\geq0,$ and $\delta=\pm(r^{i}+r^{-i}), i=2n, n\geq1$.
Let us prove that  in these polynomials the coefficients
$\rho, \delta$ are integers and  $\rho\neq\pm1, \delta\not\in \{-3, -1\}$.
Indeed, the diagonal matrices are conjugated  to  matrices $[\overline{\varphi}]^{\otimes k}, k\geq1,$ with integers coefficients, their traces are integers. Hence, for the diagonal matrices
$$
\mathrm{diag}(r, -r^{-1})~~\mbox{and}~~ \mathrm{diag}(r^{2}, -1, -1, r^{-2}),
$$
 which are conjugated  to $[\overline{\varphi}]$ and $[\overline{\varphi}]^{\otimes2}$, respectively,  we get that the sums  $r-r^{-1}$ and $r^{2}+r^{-2}$ are integers. Further, in the diagonal matrix which is conjugated  to  $[\overline{\varphi}]^{\otimes k}, k\geq1,$ the pair  $r^{k}, -r^{-k}$  or $r^{k}, r^{-k}$, in depending on the parity of  $k$, occurs only once. By induction hypothesis all other pairs with $i<k$ give the integer sums. Hence, the last unique pair for  $i=k$ gives the integer sum. Hence, all  $\rho, \delta$ are integers.

Let us prove that the coefficients  satisfy the need restrictions, i.e.  $\rho\neq\pm1, \delta\not\in \{-3, -1\}$. By assumption $|r-r^{-1}|>1$.  At first consider the coefficient  $\rho$. Without loss of generality, we can assume that  $|r|>1$. This implies the inequalities
$$
|r|^{i}>\frac{(1+|r|)^{i}}{|r|^{i}}>\frac{1+|r|^{i}}{|r|^{i}}.
$$
So  $r^{i}-r^{-i}>1$ if $r>1$. In these cases  $i$ are odd, hence the inequality  $r<-1$
implies $r^{i}-r^{-i}<-1$. So $\rho=\pm(r^{i}-r^{-i})\neq\pm1$. For $\delta =\pm(r^{i}+r^{-i})$ the numbers  $i$ are even. We proved that  $\rho=\pm(r^{i}-r^{-i})=\pm2, \pm3,\dots,$ $ i=2n+1, n\geq0.$  Since
$$
r^{\pm1}=\frac{-\rho\pm\sqrt{\rho^{2}+4}}{2},
$$
we have
$$
r^{2n}\geq \left( \frac{2+\sqrt{2^{2}+4}}{2}\right)^{2n}=(1+\sqrt{2})^{2n}>3^{n}.
$$
 Hence,  $\delta=\pm(r^{i}+r^{-i})\not\in \{-3, -1\},$ for  $i=2n, n\geq1$.

To prove residually nilpotence of $\overline{G}_k=(F_2^{ab})^{\otimes k}\rtimes_{\overline{\varphi}_{k}}\mathbb{Z}, k\geq1$
we can use Proposition
 \ref{MFcri}. Irreducible polynomials of the characteristic polynomials  $P_{\overline{\varphi}_{k}}(x)$ are polynomials of the form  $f_{\rho}(x)=x^{2}+\rho x-1$ or $f_{\delta}(x)=x^{2}+\delta x+1$, where $\rho, \delta$ are integers and $\rho\neq\pm1, \delta\not\in \{-3, -1\}$, or $p(x)=x\pm1$. It is clear that  $f_{\rho}(1)=\rho\neq\pm1$, $f_{\delta}(1)=2+\delta\neq\pm1$ and $p(1)\neq\pm1$.

 By Proposition  \ref{MFcri} all groups  $\overline{G}_k,$  $k\geq1,$ are residually nilpotent.  By Theorem \ref{Mres} we get $\gamma_{\omega^{2}}(G)=1$.
\end{proof}

We shall use the following claim.

\begin{lemma}\label{kruhcG} Let $G=F_{2}{\rtimes}_\varphi \mathbb{Z}$, $w\in\gamma_{i}(F_{2})$
and $w^{n}\in[\gamma_{i}(F_{2}), G]$, then
$$[w, g_{1},\dots, g_{s}]^{n}\in[\gamma_{i}(F_{2}),   _{s+1} G],$$ for all $s\geq1$  and $g_{i}\in G, i=1,\dots,s$.
\end{lemma}
\begin{proof} We can assume that for  $s=0$ the commutator $[w, g_{1},\dots, g_{s}]$ is equal to $w$. Then by assumption  $w^{n}\in[\gamma_{i}(F_{2}), G]$, i.e. for $s=0$ the lemma is true. 

 Let
$$
[w, g_{1},\dots, g_{s}]^{n}\in[\gamma_{i}(F_{2}),  _{s+1} G].
$$
 We will prove that  $$[w, g_{1},\dots, g_{s}, g_{s+1}]^{n}\in[\gamma_{i}(F_{2}),   _{s+2} G].$$
In this case  $$[[w, g_{1},\dots, g_{s}]^{n}, g_{s+1}]\in[\gamma_{i_{0}}(F_{2}), _{s+2} G].$$
Further, from the commutator identity  $[uv, w]=[u, w][u, w, v][v, w]$ follows
$$[[w, g_{1},\dots, g_{s}]^{n}, g_{s+1}]=[w, g_{1},\dots, g_{s}, g_{s+1}][w, g_{1},\dots, g_{s}, g_{s+1}, [w, g_{1},\dots, g_{s}]^{n-1}]\cdot$$
$$\cdot[w, g_{1},\dots, g_{s}, g_{s+1}][w, g_{1},\dots, g_{s}, g_{s+1}, [w, g_{1},\dots, g_{s}]^{n-2}]\cdots$$
$$\cdots[w, g_{1},\dots, g_{s}, g_{s+1}][w, g_{1},\dots, g_{s}, g_{s+1}, [w, g_{1},\dots, g_{s}]][w, g_{1},\dots, g_{s}, g_{s+1}],$$
that is
$$[[w, g_{1},\dots, g_{s}]^{n}, g_{s+1}]\equiv[w, g_{1},\dots, g_{s}, g_{s+1}]^{n}\mod[\gamma_{i}(F_{2}),  _{s+2} G].$$
Hence $[w, g_{1},\dots, g_{s}, g_{s+1}]^{n}\in[\gamma_{i}(F_{2}), _{s+2} G],$ and the induction hypothesis implies the need claim.

\end{proof}

It is well known that if a group  $G$ is generated by a set  $M$, then its terms of the lower central series  $\gamma_{i}(G), i\geq1,$ are generated by the simple commutators of weight  $i$ on elements of  $M$ and the next term of the lower central series $\gamma_{i+1}(G)$ \cite[Section 17]{KM}.
For example, in  $F_{2}=\langle x, y\rangle$  the term  $\gamma_{i}(F_{2}), i\geq2,$  is generated by simple commutators  $[x, y, g_{3},\dots, g_{i}]$ and $[y, x, g_{3},\dots, g_{i}],$ where  $g_{3},\dots, g_{i}\in\{x, y\}$, and by the next term  $\gamma_{i+1}(F_{2})$. From these observations follows

\begin{lemma}\label{centrG} Let $G=F_{2}{\rtimes}_\varphi \mathbb{Z}$, $F_{2}=\langle x, y\rangle$
and $\mathbb{Z}=\langle t\rangle$, then the subgroup
$$
[\gamma_{i}(F_{2}), _{s} G],~~~  i\geq2, s\geq1,
$$  is generated by the simple commutators  $$[x, y, g_{3},\dots, g_{i+1},\dots, g_{+s}]  \  \text{and} \ [y, x, g_{3},\dots, g_{i+1},\dots, g_{i+s}],$$ where $g_{j}\in \{x, y\},$ for $j\leq i,$ $g_{j}\in \{x, y, t\},$ for $j>i,$ and the subgroup $[\gamma_{i}(F_{2}), _{s+1}G]$.
\end{lemma}

Now we are able to fined  the length of the lower central series for groups from  Theorem \ref{w2}.
\begin{theorem}\label{lw2}
Let $G = F_2 {\rtimes}_\varphi \mathbb{Z}$.
If $\mathrm{det}
[\overline{\varphi}]=-1$ and $\mathrm{tr}[\overline{\varphi}]$ is an odd number,  $\mathrm{tr}[\overline{\varphi}]\neq\pm1$, then the length of the lower central series of  $G$ is equal to $\omega^{2}$.
\end{theorem}
\begin{proof} By Theorem \ref{w2} we have $\gamma_{\omega^{2}}(G)=1$. Also, from the same theorem follows that all extensions $\gamma_{i}(F_{2})/\gamma_{i+1}(F_{2})\rtimes_{\widehat{\varphi}_{i}}\mathbb{Z},$ $i\geq1,$
where $\widehat{\varphi}_{i}$ is the automorphism of  $\mathbb{Z}$-module
$\gamma_{i}(F_{2})/\gamma_{i+1}(F_{2})$ that is induced by the automorphism  $\varphi$, is residually    nilpotent.

Hence, by Lemma \ref{Nilw2} it is enough  to prove  that $\gamma_{\omega}(G)$ is non-trivial.

The quotient $\gamma_{2}(F_{2})/\gamma_{3}(F_{2})\simeq\mathbb{Z}$ is generated by the image of the commutator  $[x, y]$. Since $\mathrm{det}[\overline{\varphi}]=-1$, the induced automorphism  $\widehat{\varphi}_{2}=-id$. Hence
$$[[x, y], t]=[x, y]^{-1}[x, y]^{t}\equiv[x, y]^{-2}\mod\gamma_{3}(F_{2}).$$
So, $$[x, y]^{2}\equiv[[x, y], t]^{-1}\mod\gamma_{3}(F_{2})$$ and $[x, y]^{2}\in[\gamma_{2}(F_{2}), G]<\gamma_{3}(G)$. By Lemma \ref{kruhcG} we have inclusions $$
[x, y, g_{1},\dots, g_{s}]^{2}, ~~~[y, x, g_{1},\dots, g_{s}]^{2}\in[\gamma_{2}(F_{2}), _{s+1} G], ~~s\geq1,
$$
 where $g_{i}\in\{x, y, t\}$. Further, by Lemma \ref{centrG} the simple commutators
$$[x, y, g_{1},\dots, g_{s}], ~~[y, x, g_{1},\dots, g_{s}], ~~s\geq1,$$  where $g_{i}\in\{x, y, t\}$, generate by modulo of the subgroup
$[\gamma_{2}(F_{2}), ~ _{s+1}G]$ a  finitely generated abelian group with identity  $g^{2}=1$. Using induction by $k\in\mathbb{N},$ it is easy  to check that $[x, y]^{2^{k}}\in[\gamma_{2}(F_{2}),  ~ _{k}G] \leq \gamma_{k+2}(G)$.

By theorem assumption the characteristic polynomial  $P_{\overline{\varphi}}(x)$ is irreducible over  $\mathbb{Z}$ and  $P_{\overline{\varphi}}(1)=-\mathrm{tr}[\overline{\varphi}] \not= 0$. By Proposition \ref{MFcri} the group $F^{ab}_{2}\rtimes_{\overline{\varphi}}\mathbb{Z}$ is residually   $p$-finite for any prime divisor of   $\mathrm{tr}[\overline{\varphi}]$.  Lemma \ref{Pphi1} implies
$$
F^{ab}_{2} \geq \gamma_k(F^{ab}_{2}\rtimes_{\overline{\varphi}}\mathbb{Z}) \geq P^{k-1}_{\overline{\varphi}}(1) F^{ab}_{2}.
$$
Hence, the index
$$
|F^{ab}_{2} : \gamma_k(F^{ab}_{2}\rtimes_{\overline{\varphi}}\mathbb{Z})|
$$
divides $P^{2(k-1)}_{\overline{\varphi}}(1)$.  More precisely,   the images of  $x,  y$ in $F^{ab}_{2}$, under multiplication on  $m^{l}$,  lie in   $\gamma_{l+1}(F^{ab}_{2}\rtimes_{\overline{\varphi}}\mathbb{Z})$. Hence, we have inclusions $$x^{m},   y^{m}\in[F_{2}, t]\gamma_{2}(F_{2}) \leq [F_{2}, G] \leq \gamma_{2}(G).
$$
Lemma \ref{kruhcG} implies
$$
[ x, g_{1},\dots, g_{s}]^{m}, ~~[y, g_{1},\dots, g_{s}]^{m}\in[F_{2}, ~_{s+1}G],~~~ s\geq1.
$$
 In particular,  $[x, y]^{m}\in[F_{2}, G, G] \leq \gamma_{3}(G)$.
By Lemma \ref{centrG} the simple commutators
$[x,  g_{1},\dots, g_{s}],$ $ [y,  g_{1},\dots, g_{s}],$  $s\geq1,$  where $g_{i}\in\{x, y, t\}$, generate by modulo  $[F_{2}, _{s+1} G]$ a finitely generated abelian group with identity  $g^{m}=1$.  By induction  $[x, y ]^{m^{k}}\in\gamma_{k+2}(G), k\in\mathbb{N}.$

Since $m$ and $2$ are mutually prime, we have $[x, y]\in\gamma_{\omega}(G)$. Hence $\gamma_{2}(F_{2})\leq\gamma_{\omega}(G)$.  The theorem is proved.
\end{proof}

\noindent\textit{Proof of Theorem\ref{NAF2Z}.} Follows from Theorems \ref{F2Zr1}, \ref{F2Z2r} and \ref{lw2}.\vspace{2ex} \hfill {\large $\Box$}

\noindent\textit{Proof of Theorem \ref{F2Zc}.} From the proved results 
follows that for the lower central series of   $G = F_2 {\rtimes}_\varphi \mathbb{Z}$ there exist only three possibilities:  $\gamma_{\omega}(G)=\gamma_{2}(G)=F_{2}$ (see Proposition \ref{F2Znr}); or $\gamma_{\omega}(G)=1$ (see Proposition \ref{F2Zr} and Theorem \ref{F2Z2r}); or the length of the lower central series is equal to  $\omega^{2}$  (see Theorems \ref{w2},   \ref{lw2}). \hfill {\large $\Box$}

\section*{Acknowledgments}
Authors are grateful to participants of the seminar ``\'{E}variste Galois'' at
Novosibirsk State University and personally to Yu. I. Sosnovsky for fruitful discussions.

The research  is supported by Ministry of Science and Higher Education of Russia
(agreement No. 075-02-2020-1479/1).


\begin{thebibliography}{HD}

\bibitem{AF}
M. Aschenbrenner, S.  Friedl,
\textit{Residual properties of graph manifold groups} (English summary)
Topology Appl.,  158, no. 10 (2011),  1179--1191.

\bibitem{Azar}
  D. N. Azarov,
 \textit{On the virtually $p$-residual finiteness}, (Russian) Chebyshevskii Sb., 11, no. 3 (2010), 11--20.



\bibitem{BBr}
V. G. Bardakov, O. V. Bryukhanov,
\textit{On linear representations of some extensions},
Vestnik Novosibirsk Univ. Ser. Mat. Mekh. Inf., 7
(2007), 45--58 (Russian).

\bibitem{BM}
V. G. Bardakov, R. V. Mikhailov,
\textit{On the residual properties of link groups}, (Russian)
Sibirsk. Mat. Zh. 48 (2007), no. 3, 485--495;
translation in Siberian Math. J. 48 (2007), no. 3, 387--394.

\bibitem{BMVW} V. G. Bardakov, R. Mikhailov, V. V. Vershinin and J. Wu,
On the pure virtual braid group $PV_3$,  Commun. in Algebra, 44, no. 3 (2016), 1350--1378.

\bibitem{BMN}
Bardakov V. G., Mikhalchishina Yu. A., and Neshchadim M. V., Virtual link groups,
Siberian Mathematical Journal, 58, no. 5 (2017), 765--777.

\bibitem{BN-2}
 Bardakov V. G., Neshchadim M. V., Knot Groups and Residual Nilpotence, Proceedings of the
Steklov Institute of Mathematics, 304,  no. 1 (2019),   23--30.

\bibitem{BNN}
 Bardakov V. G., N. Nanda, and Neshchadim M. V., On the lower central series of some virtual
knot groups, Journal of Knot Theory and Its Ramifications, 29, no. 9 (2020),

\bibitem{BN-4}
V. G. Bardakov, M. V. Neshchadim, \textit{On the lower central  series of Baumslag-Solitar groups},
 (Russian) Algebra Logika, 59,  no. 4 (2020),  413--431.


\bibitem{Bau}
G. Baumslag,
\textit{On the residual nilpotence of certain one-relator groups},
Comm. Pure Appl. Math., 21, no. 5 (1968), 491--506.

\bibitem{Bau-1}
G. Baumslag,
\textit{Finitely generated cyclic extensions of free groups are residually finite},
Bull. Austral. Math. Soc., 5, no. 1 (1971), 87--94.



\bibitem{Br}
O. V. Bryukhanov,
\textit{Nilpotent approximability of fundamental groups of compact three-dimensional Sol-manifolds}.
(Russian) Sibirsk. Mat. Zh. 57, no. 2 (2016),  247--258;
translation in Sib. Math. J., 57,  no. 2  (2016), 190--199.

\bibitem{Br-1}
O. V. Bryukhanov,
\textit{Approximation properties and linearity of groups},
J. Math. Sci. (N.Y.) 188, no. 4 (2013),  354--358.


\bibitem{GL}
E. A. Gorin, V. Ja. Lin,
\textit{Algebraic equations with continuous coefficients, and certain questions of the algebraic theory of braids}. (Russian)
Mat. Sb. (N.S.),  78 (120), 1969 579--610.

\bibitem{G}
K. W. Gruenberg,
\textit{Residual properties of infinite soluble groups},
Proc. Lond. Math. Soc., 1957, 7,  29--62.

\bibitem{GM}
J. Guaschi, C.de Miranda e Pereiro,
\textit{Lower central and derived series of semi-direct products,
and applications to surface braid groups},
J. Pure Appl. Algebra,  224, no. 7 (2020), 39 pp.

\bibitem{FR1} M.~Falk and R.~Randell,
\textit{The lower central series of a fiber type arrangement},
Invent. Math., 82, (1985), 77--88.


\bibitem{KM}
M. I. Kargapolov, Yu. I. Merzlyakov,
\textit{Fundamentals of the Theory of Groups},
Springer-Verlag, New York; Heidelberg; Berlin (1979).

\bibitem{Kh}
E. I. Khukhro, Nilpotent Groups and Their Automorphisms,  Walter de Gruyter, 1993.



\bibitem{L}
E. D. Loginova,
\textit{Residual finiteness of the free product of two groups with commuting subgroups},
Siberian Math. J., 40, no. 2 (1999), 341--350.


\bibitem{Mal}
 A. I. Malcev,
 \textit{On homomorphisms onto finite groups},
 Uchen. Zapiski Ivanovsk. ped. instituta, 18, no. 5 (1958), 49--60
 (also in Selected papers, Vol. 1, Algebra,
1976, 450--462) (Russian).


\bibitem{Mal-1}   A. I. Malcev,
\textit{Generalized nilpotent algebras and their associated groups},
Mat. Sbornik N.S., 25, no. 3 (1949), 347--366.


\bibitem{Mal-2}
A. I. Malcev,
\textit{Isomorphic representation of infinite groups by matrices},
Mat. Sb., 8, no. 3 (1940), 405-422.


\bibitem{McC}
J. McCarron,
\textit{Residually nilpotent one-relator groups with nontrivial centre},
Proc. Amer. Math. Soc., 124, no. 1 (1996), 1--5.


\bibitem{Mer}
Yu. I. Merzlyakov, \textit{Rational Groups} (Russian), Nauka, Moscow (1987).


\bibitem{M-1}
R.~Mikhailov,
\textit{On nilpotent and solvable residual finiteness of groups},
Mat. Sb., 196, no. 11 (2005), 109--126.


\bibitem{MP} R.~Mikhailov and I.~B.~S.~Passi,
\textit{Lower Central and Dimension Series of Groups},
Lecture Notes in Mathematics, 1952, Springer-Verlag Berlin Heidelberg, 2009.


\bibitem{M}
R.~Mikhailov,
\textit{A one-relator group with long lower central series},
Forum Math.,  28, no. 2 (2016),  327--331.



\bibitem{Mol1}
D. I. Moldavanskii,
\textit{ On $p$-residually finiteness  of  HNN-extensions}, (Russian)
Vestnik Ivanov. Gos. Univ.,  3 (2000), 129--140.


\bibitem{Mol2}
D. I. Moldavanskii,
\textit{On the intersection of subgroups of finite index in the Baumslag-Solitar groups},
(Russian) Mat. Zametki 87 (2010),  1, 92--100;
translation in Math. Notes,  87,  1-2 (2010),  88--95.



\bibitem{Mol4}
D. I. Moldavanskii,
\textit{Residual Nilpotence of Groups with One Defining Relation},
(Russian) Mat. Zametki, 107, 5 (2020),  752--759.




\bibitem{Mol3}
D. Moldavanskii,
\textit{On the residual properties of Baumslag-Solitar groups},
Communications in Algebra, 46,  no. 9 (2018), 3766--3778.

\bibitem{S}
E. V. Sokolov, \textit{Certain residual properties of generalized Baumslag-Solitar groups}, J. Algebra,  582 (2021), 1--25.




\end{thebibliography}
\end{document}